\documentclass[reqno, 12pt,dvipsnames ]{amsart}
\usepackage{amsmath,amsxtra,amssymb,latexsym, amscd,amsthm}
\usepackage{graphicx,color}
\usepackage[utf8]{inputenc}
\usepackage[mathscr]{euscript}
\usepackage{mathrsfs}
\usepackage[english]{babel}
\usepackage{enumerate}
\usepackage{color}
\usepackage{arcs}
\usepackage{tikz}
\usepackage{float}
\usetikzlibrary{calc}
\usetikzlibrary{math}
\usepackage{graphicx}
\usepackage{pdfpages}
\usepackage[toc,page]{appendix}
\usepackage{enumitem}
\usepackage{hyperref}

\newtheorem {theorem}{Theorem}[section]
\newtheorem {corollary}[theorem]{Corollary}
\newtheorem {proposition}[theorem]{Proposition}
\newtheorem {lemma}[theorem]{Lemma}
\newtheorem {claim}[theorem]{Claim}

\newtheorem{definition}[theorem]{Definition}
\newtheorem {example}{Example}[section]
\newtheorem{remark}{Remark}[section]

\newcommand{\dist}{{\rm dist}}

\newcommand{\Int}{{\rm Int}}

\newcommand{\Ext}{{\rm Ext } }

\newcommand{\ord}{{\rm  ord }}

\newcommand{\ak}{\mathfrak}

\newcommand{\bb}{\mathbb}

\title{PLANAR EXTENSIONS IN O-MINIMAL STRUCTURES}
\author{Si Tiep Dinh}
\address{Institute of Mathematics, VAST, 18, Hoang Quoc Viet Road, Cau Giay District 10307, Hanoi, Vietnam}
\email{dstiep@math.ac.vn}

\author{Nhan Nguyen}
\address{FPT University, Danang, Vietnam}
\email{nguyenxuanvietnhan@gmail.com}

\date{\today}

\begin{document}

\begin{abstract}
Let $X \subset \mathbb{R}^2$ be a closed definable set of dimension at most $1$, and let
$h : X \to \mathbb{R}^2$ be a definable continuous injective map. In this paper, we establish necessary and sufficient combinatorial conditions, formulated in terms of cyclic orders at topological singular points and orientations of Jordan curves, for $h$ to admit a definable homeomorphic extension to the whole plane.
\end{abstract}

\maketitle

\section*{Introduction}
The classical Jordan curve theorem states that every simple closed curve (Jordan curve) in the plane $\mathbb{R}^2$ separates the plane into exactly two connected components: a bounded interior and an unbounded exterior. The Schoenflies theorem strengthens this statement by asserting that any homeomorphism from the unit circle onto a Jordan curve extends to a homeomorphism of the entire plane, mapping the interior to the interior and the exterior to the exterior. Consequently, any two embeddings of a Jordan curve in the plane are ambiently homeomorphic.

A natural question is whether a similar extension property remains valid for more general subsets of the plane. More precisely, given an embedding (i.e., a continuous injective map)
$$
h \colon X \hookrightarrow \mathbb{R}^2,
$$
where $X \subset \mathbb{R}^2$, one may ask:

\begin{quote}
\emph{When does $h$ extend to a homeomorphism of the whole plane?}
\end{quote}

In general, even embeddings of simple sets need not admit such extensions. For instance, let $X$ be the union of a circle and its center, and let $h$ be the identity on the circle while mapping the center to a point outside the circle. Then $h$ cannot extend to a homeomorphism of the plane. Thus, any positive extension theorem must impose additional conditions on the embedding $h$.

This question is closely related to classical problems in topological graph theory. Two embeddings of a graph $G$ into a surface $S$ are said to be equivalent if there exists a homeomorphism of $S$ mapping one embedding to the other. Whitney~\cite{Whitney32} proved that every $3$-connected planar graph has a unique embedding in the sphere up to equivalence. Many further results in this direction can be found in the book of Mohar and Thomassen~\cite{Mohar03}.

In~\cite{Youngs48}, Young gave a partial answer to the extension problem: if $M$ and $N$ are $2$-dimensional manifolds whose boundaries are unions of pairwise disjoint Jordan curves, then a homeomorphism
$
h:\partial M\to \partial N
$
extends to a homeomorphism $M\to N$ if and only if it carries a concordant orientation on $\partial M$ to a concordant orientation on $\partial N$. In the planar case, this generalizes the Schoenflies theorem from a single Jordan curve to finitely many pairwise disjoint Jordan curves.

A much more general extension theorem was proved by Gehman \cite{Gehman1} for continuous
one-to-one correspondences between plane continuous curves. He proved that if $M, M'\subset \mathbb{R}^2$ are plane continuous curves and $T: M\to M'$ is a homeomorphism such that every simple closed curve $\gamma\subset \mathbb{R}^2$ containing an arc $AB\subset M$, there exists a simple closed curve $\gamma'\subset \bb R^2$ containing $h(AB)$ and satisfying 
$$ h(M \cap \Int(\gamma)) = M' \cap \Int (\gamma')$$
and symmetrically with the role of $M$ and $M'$ reversed, then $T$ extends to a homeomorphism $U: \mathbb{R}^2 \to \mathbb{R}^2$.  Here, $\Int(\gamma)$ denotes the interior of $\gamma$.

Later, in~\cite{Gehman2}, Gehman extended his result to $E$-sets, which may have infinitely many connected components, by imposing compatible side-preserving and interior-preserving conditions. While Gehman's theorems are topologically natural and very general, their hypotheses are often difficult to verify in practice, even for tame embeddings.

In this paper, we study the extension problem in the o-minimal setting, which provides a rich framework for tame geometry. Let $X \subset \mathbb{R}^2$ be a closed definable set of dimension at most $1$, and let
$
h \colon X \to \mathbb{R}^2
$
be a definable embedding. When $X$ is unbounded, we additionally assume that $h$ is coercive, i.e.,
$$
\|x\| \to \infty
\quad \Longrightarrow \quad
\|h(x)\| \to \infty.
$$
This growth condition is necessary for the existence of a homeomorphic extension to $\mathbb{R}^2$.

Our main result provides necessary and sufficient combinatorial conditions for the existence of a definable ambient extension. In the compact case, Theorem~\ref{thm_main} shows that $h$ admits a definable homeomorphic extension to $\mathbb{R}^2$ if and only if there exists $c\in\{\pm1\}$ such that

\begin{enumerate}
    \item $\ord(h,p)=c$ for every topological singular point
    $p\in X$ of degree at least $3$;
    
    \item $w(h,\gamma)=c$ for every facial Jordan curve
    $\gamma$ of $X$;
    
    \item for every facial Jordan curve $\gamma$ of $X$ and every
    $x\in X\setminus \gamma$,
    $$
    x\in \Int(\gamma)
    \quad\Longleftrightarrow\quad
    h(x)\in \Int(h(\gamma)).
    $$
\end{enumerate}

Here, a facial Jordan curve is a Jordan curve contained in the boundary of a bounded face of $X$ and enclosing that face; Lemma~\ref{graph-properties} (viii) shows that such a curve exists and is unique for each bounded face.  The notation $\ord(h,p)$ indicates whether the local cyclic order of the branches of $X$ at the singular point $p$ is preserved or reversed: $\ord(h,p)=1$ means that $h$ preserves this cyclic order, while $\ord(h,p)=-1$ means that it reverses it. Similarly, $w(h,\gamma)$ indicates whether $h$ preserves or reverses the orientation of the Jordan curve $\gamma$.

The proof overcomes the main difficulty that the complementary faces of $X$ need not have Jordan curve boundaries, preventing a direct face-by-face application of the Schoenflies theorem. We resolve this issue by adding finitely many suitable definable arcs to $X$, chosen so that the enlarged set decomposes the plane into regions bounded by Jordan curves. Conditions~(1)--(3) guarantee that each added arc in the source corresponds to an admissible arc in the target, allowing $h$ to extend to a definable embedding of the enlarged sets. The enlarged source and target then admit compatible Jordan decompositions to which a definable version of the Schoenflies theorem can be applied region by region.

The unbounded case is reduced to the compact one via a compactification argument using inversions. We introduce the notion of cyclic orders at infinity for the unbounded branches of $X$. After choosing points $p$ and $q$ near infinity (outside $X$ and $h(X)$, respectively), we consider the compactified map $g_{p, q}: \phi_p(X) \cup \{0\} \to \mathbb R^2$  given by 
$$
g_{p,q}(x)=\left\{
\begin{array}{rl}
\phi_q\circ h\circ \phi_p^{-1}(x), & x\neq 0,\\
0, & x=0,
\end{array}\right.
$$
where
$$
\phi_p(x)=\frac{x-p}{\|x-p\|^2}
$$
(and similarly for $\phi_q$). Theorem~\ref{thm_main_unbounded} shows that $h$ extends if and only if $g_{p,q}$ does.

From the perspective of topological graph theory, our results give an ambient equivalence criterion for finite plane graphs. Here a plane graph means a graph embedded in $\mathbb R^2$. We allow loops, multiple edges and unbounded edges, and we assume that all edges are simple polygonal arcs. It is known that every planar graph admits a realization in $\mathbb R^2$ whose edges are simple polygonal arcs; see, for example, \cite[Lemma 2.1.1]{Mohar03}.

More precisely, let $f:G_1\to G_2$ be a homeomorphism between finite plane graphs. Then $f$ extends to a homeomorphism of the whole plane $\mathbb R^2$ if and only if Conditions (1)--(3) hold. Moreover, if $f$ is definable then the extension can also be chosen to be definable. 

The paper is organized as follows. Section \ref{section1} recalls the necessary notions of orientation, local cyclic order, and definable Jordan curves. Section \ref{section2} develops some elementary properties of plane graphs and facial Jordan curves. Section \ref{section3} proves the extension theorem in the compact case. Section \ref{section4} treats the unbounded case by introducing cyclic orders at infinity and reducing the problem to the compact case via inversion. Section \ref{section5} presents some examples showing that Conditions (1)--(3) cannot be omitted.

Throughout this paper, we fix an o-minimal structure on $(\mathbb R,+,\cdot)$, and all definability is understood with respect to this structure. We assume some familiarity with the basic notions of o-minimality. However, the arguments in this paper use only elementary properties of o-minimal structures and do not require any advanced machinery. For further background, we refer the reader to \cite{Coste2000, Dries1998, Loi1}.

\subsection*{Notations}
\begin{itemize}
    \item For $X\subset \mathbb R^2$, we denote by $\partial X$ and $\overline{X}$ respectively the topological boundary and the closure of $X$ in $\mathbb R^2$.

    \item Given $p\in \mathbb R^2$, we denote by $\mathbb B_r(p)$ and $\mathbb S_r(p)$
    the open ball and the circle of radius $r$ centered at $p$,
    respectively. When $p=0$, we write $\mathbb S_r$ instead of
    $\mathbb S_r(0)$.
\end{itemize}
\subsection*{Acknowledgments}
Part of this work was carried out during the authors' visits to the Vietnam Institute for Advanced Study in Mathematics (VIASM). The authors gratefully acknowledge VIASM for its warm hospitality and generous support. The second author was supported by the Vietnam Ministry of Education and Training (MOET) under grant number B2025-CTT-01.

\section{Auxiliary results}\label{section1}

\subsection{Orientations of arcs and Jordan curves}

We first fix some notation. By an \emph{arc} we mean the image of an injective continuous map
$\alpha:[0,1]\to \mathbb R^2$.
If the endpoints of the arc are $p$ and $q$, we denote the arc by
$[p,q]$, and its relative interior by
$\mathring{\alpha}=(p,q)$.
If the arc is contained in a set $X$, we may write $[p,q]_X$.

Let $Y\subset \mathbb R^2$. An arc $\alpha$ is said to be
\emph{internally contained in $Y$} if
$\mathring{\alpha}\subset Y$.

A \emph{Jordan curve} is the image of a continuous map
$\gamma:[0,1]\to \mathbb R^2$ such that $\gamma(0)=\gamma(1)$ and
$\gamma|_{[0,1)}$ is injective. By the Jordan curve theorem,
$\mathbb R^2\setminus \gamma$ has exactly two connected components. We
denote the bounded component by $\Int(\gamma)$ and the unbounded component
by $\Ext(\gamma)$.

If $X\subset \mathbb R^2$ is a definable set of dimension at most $1$, we
denote by $\mathfrak J(X)$ the set of Jordan curves contained in $X$.

An orientation of an arc is obtained by choosing one endpoint as the initial point and the other as the terminal point. Thus an arc has two possible orientations, corresponding to its two directions.
A Jordan curve also has two possible orientations. We say that an oriented Jordan curve is \emph{counterclockwise} if its interior lies locally on the left-hand side as the curve is traversed, and \emph{clockwise} if its interior lies locally on the right-hand side.

Let $\gamma$ be an oriented arc or a Jordan curve, and let
$p\in \mathring{\gamma}$. Then $p$ divides a neighborhood of $p$ in
$\gamma$ into two oriented subarcs having $p$ as a common endpoint. The subarc for which $p$ is the terminal point is called the \emph{incoming arc} of $\gamma$ at $p$, while the other subarc is called the \emph{outgoing arc} of $\gamma$ at $p$. When the arc $\gamma$ and the point $p$ are clear from the context, we simply call them the incoming and outgoing arcs.

Let $\gamma$ be an oriented Jordan curve, and let $p,q\in \gamma$ be
distinct points. We denote by
$[p,q]^+_\gamma$
the subarc of $\gamma$ from $p$ to $q$ with respect to the
counterclockwise orientation of $\gamma$. We denote by
$[p,q]^-_\gamma$
the other subarc of $\gamma$ from $p$ to $q$. We use the analogous
notation for open and half-open subarcs.

Finally, let $h:\gamma\to \mathbb R^2$ be a continuous injective map, where $\gamma$ is a Jordan curve. Then $h(\gamma)$ is again a Jordan curve. We define
$$
w(h,\gamma)=
\left\{
\begin{array}{rl}
1, & \text{if } h \text{ preserves the orientation of } \gamma,\\
-1, & \text{if } h \text{ reverses the orientation of } \gamma.
\end{array}
\right.
$$

\begin{lemma}\label{PreserveReverse}
Let $H:\mathbb R^2\to \mathbb R^2$ be a homeomorphism. Then either $H$
preserves the orientation of every Jordan curve in $\mathbb R^2$, or $H$
reverses the orientation of every Jordan curve in $\mathbb R^2$.
\end{lemma}

A standard topological result guarantees the global consistency of this factor.

To compare two distinct oriented Jordan curves
$\gamma_1, \gamma_2 \in \mathfrak{J}(\mathbb{R}^2)$, we define the
\emph{relative orientation operator}
$$
o(\gamma_1,\gamma_2)=
\left\{
\begin{array}{rl}
1, & \text{if } \gamma_1 \text{ and } \gamma_2
\text{ possess the same orientation,}\\
-1, & \text{otherwise.}
\end{array}
\right.
$$

The following lemma characterizes how local intersections determine the global behavior of overlapping Jordan curves.

\begin{lemma}\label{InOut}
Let $\gamma_1,\gamma_2\subset \bb R^2$ be oriented definable Jordan curves such that
$\gamma_1\cap \gamma_2$ contains an arc $\alpha$. Then exactly one of the following holds:
\begin{enumerate}
\item [(i)] $\alpha\subset \overline{\Int(\gamma_1)\cap \Int(\gamma_2)}$;
\item [(ii)] $\alpha\cap \overline{\Int(\gamma_1)\cap \Int(\gamma_2)}=\varnothing$.
\end{enumerate}
Moreover, in case \textit{(i)}, the curves $\gamma_1$ and $\gamma_2$ induce the same orientation
on $\alpha$ if and only if $o(\gamma_1,\gamma_2)=1$; in case \textit{(ii)}, they induce the same
orientation on $\alpha$ if and only if $o(\gamma_1,\gamma_2)=-1$.
\end{lemma}

\begin{proof}
Let $p\in \mathring{\alpha}$. Since $\alpha$ is a common embedded arc of
both curves, there exists a small open ball $\mathbb B_r(p)$ centered at $p$ such that
$$
\mathbb B_r(p)\cap \gamma_1 = \mathbb B_r(p)\cap \gamma_2
$$
is a single open arc, say $\beta$. This arc $\beta$ divides the ball $\mathbb B_r(p)$
into two open connected components, which we denote by $B^+$ and $B^-$.
For each $i=1,2$, exactly one of the two components $B^+$ or $B^-$ lies in $\Int(\gamma_i)$.
Since the curves are continuous and $\alpha$ is connected, the choice of which side contains the interior cannot suddenly switch along $\alpha$.
Therefore, there are exactly two global possibilities along the whole arc
$\alpha$:
\begin{enumerate}
\item [(i)] The interiors of $  \gamma_1  $ and $  \gamma_2  $ lie on the \emph{same} side of $  \alpha  $. In this case, we have $$\alpha \subset \overline{\Int(\gamma_1) \cap \Int(\gamma_2)}.$$
\item [(ii)] The interiors of $  \gamma_1  $ and $  \gamma_2  $ lie on \emph{opposite} sides of $  \alpha  $. In this case, $$\alpha \cap \overline{\Int(\gamma_1) \cap \Int(\gamma_2)} = \varnothing.$$
\end{enumerate}
This shows that exactly one of the two cases holds.
It remains to relate these cases to the orientations induced on $\alpha$.
Recall that a curve is oriented counterclockwise if and only if its interior lies locally on the left-hand side when traversing the curve.

\noindent\textbf{Case (i):} The interiors lie on the same side of $  \alpha  $.
Since both curves have their interior on the same side, they must be traversed in the same direction along $  \alpha  $ in order for the interior to remain on the left for both. Therefore, $  \gamma_1  $ and $  \gamma_2  $ induce the same orientation on $  \alpha  $ if and only if they have the same global orientation, that is, $  o(\gamma_1, \gamma_2) = 1  $.

\noindent\textbf{Case (ii):} The interiors lie on opposite sides of $  \alpha  $.
For the interior to stay on the left of each curve, the two curves must now be traversed in {opposite} directions along $  \alpha  $. Hence, $  \gamma_1  $ and $  \gamma_2  $ induce the same orientation on $  \alpha  $ if and only if they have opposite global orientations, that is, $  o(\gamma_1, \gamma_2) = -1  $.
\end{proof}

Let $\gamma$ be a Jordan curve and let $P = \{p_1, \dots, p_n\} \subset \gamma$ ($n \geq 3$) be a set of pairwise distinct points.  We define a \textit{cyclic relation} $<_\gamma$ on $P$ as follows: for $i \neq j$, we write$$p_i <_\gamma p_j$$if $p_j$ is the first point among $P \setminus \{p_i\}$ encountered when traversing $\gamma$ counterclockwise starting from $p_i$. Equivalently, the open counterclockwise subarc $(p_i, p_j)^+_\gamma$ contains no points in $P$.

Now, let $h: \gamma \to \mathbb{R}^2$ be a definable embedding. By the Schoenflies theorem, $h$ extends to a global homeomorphism of $\mathbb{R}^2$. Applying Lemma \ref{PreserveReverse} to this extension, we obtain:

\begin{lemma}\label{cyclic-order-lemma} The map $h$ either preserves the cyclic order of any finite set of points $p_1, \dots, p_n \in \gamma$ or reverses it.\end{lemma}

Let $X$ be a definable set of dimension $1$. A point $p\in X$ is called a
\emph{topological regular point} if there exists a definable open neighborhood $U_p$ of $p$
in $\bb R^2$ such that $X\cap U_p$ is definably homeomorphic to an open interval in $\bb R$.

 Fix a point $p \in X$. By Hardt's triviality theorem \cite{Coste2000, Hardt1980, Dries1998}, for sufficiently small $r>0$, the topological type of
$X\cap \mathbb B_r(p)\setminus \{p\}$ is constant. We denote by $\deg(p,X)$ the number of connected
components of $X\cap \mathbb B_r(p)\setminus \{p\}$ for such $r$ and call it the \emph{degree of $p$
with respect to $X$}. Clearly,
$$
\deg(p,X)=2
\quad \Longleftrightarrow \quad
p \text{ is a topological regular point of } X.
$$

Assume now that $n:=\deg(p,X)\ge 3$, and let
$$
\{X_1^p,\dots,X_n^p\}
$$
be the connected components of $X\cap \mathbb B_r(p)\setminus \{p\}$.
For each $i=1,\dots,n$, the intersection $\mathbb S_r(p)\cap X_i^p$ consists of exactly one point, which we denote by $p_i$.

To define the local cyclic order at $p$, we need the following lemma.

\begin{lemma}\label{Not-depend} 
Let $\gamma,\sigma\subset \mathbb B_r(p)$ be definable Jordan curves such that
$p\in \Int(\gamma)\cap \Int(\sigma)$ and, for each $i=1,\dots,n$, each of the
intersections $X_i^p\cap \gamma$ and $X_i^p\cap \sigma$ consists of
exactly one point, denoted by $p_i$ and $q_i$, respectively. Then, for
all $1\le i\ne j\le n$,
$$
p_i<_\gamma p_j
\quad \Longleftrightarrow \quad
q_i<_\sigma q_j.
$$
\end{lemma}

\begin{proof}
It is enough to prove one implication, since the converse follows by interchanging
the roles of $\gamma$ and $\sigma$.

Assume that $ p_i<_\gamma p_j .$
By the definition of $p_i<_\gamma p_j$, the open arc
$(p_i,p_j)^+_\gamma$ contains no point of $X$. Hence
$$
(p_i,p_j)^+_\gamma\cap X=\varnothing.
$$
Thus
$$
(p_i,p_j)^+_\gamma\subset \mathbb B_r(p)\setminus X.
$$
Since $(p_i,p_j)^+_\gamma$ is connected, it is contained in a connected component
$U$ of $\mathbb B_r(p)\setminus X$.

We first claim that if
$$
(q_i,q_j)^+_\sigma\not\subset U,
$$
then
$$
(q_i,q_j)^+_\sigma\cap U=\varnothing.
$$

Indeed, the component $U$ is the local sector determined by the two branches
$X_i^p$ and $X_j^p$ and containing the arc $(p_i,p_j)^+_\gamma$. Since
$(p_i,p_j)^+_\gamma$ contains no point $p_k$ with $k\neq i,j$, no branch
$X_k^p$, $k\neq i,j$, lies on the boundary of this sector. Therefore the
relative boundary of $U$ in $\mathbb B_r(p)$ is contained in
$$
X_i^p\cup X_j^p\cup\{p\}.
$$
Since $\sigma$ meets $X_i^p$ and $X_j^p$ exactly at $q_i$ and $q_j$, respectively,
and since $p\notin\sigma$, we get
$$
\partial U\cap \sigma=\{q_i,q_j\}.
$$
Suppose that
$$
(q_i,q_j)^+_\sigma\cap U\neq\varnothing.
$$
Since $U$ is open, the set
$$
(q_i,q_j)^+_\sigma\cap U
$$
is open in the arc $(q_i,q_j)^+_\sigma$. Moreover, the only possible boundary
points of this intersection on $\sigma$ are points of
$$
\partial U\cap\sigma=\{q_i,q_j\}.
$$
It follows that $$(q_i,q_j)^+_\sigma\cap U =(q_i,q_j)^+_\sigma$$ which 
contradicts the assumption that $(q_i,q_j)^+_\sigma\not\subset U$.
This proves the claim.

We now prove that
$$
(q_i,q_j)^+_\sigma\subset U.
$$
Suppose, by contradiction, that
$$
(q_i,q_j)^+_\sigma\not\subset U.
$$
By the claim,
$$
(q_i,q_j)^+_\sigma\cap U=\varnothing.
$$
Choose the orientations on $\gamma$ and $\sigma$ to be counterclockwise and define
$$
\gamma_1:=[p,p_i]_{X_i^p}\cup [p_i,p_j]^+_\gamma\cup [p_j,p]_{X_j^p},
$$
and
$$
\sigma_1:=[p,q_i]_{X_i^p}\cup [q_i,q_j]^+_\sigma\cup [q_j,p]_{X_j^p}.
$$
We orient $\gamma_1$ and $\sigma_1$ so that their orientations agree, respectively,
with those of
$$
[p_i,p_j]^+_\gamma
\quad\text{and}\quad
[q_i,q_j]^+_\sigma .
$$
By construction, the Jordan domain bounded by $\gamma_1$ is the sector $U$.
Thus
$$
\Int(\gamma_1)\subset U.
$$
On the other hand, since
$$
(q_i,q_j)^+_\sigma\cap U=\varnothing,
$$
the Jordan domain bounded by $\sigma_1$ lies on the opposite side of the common
branches $X_i^p$ and $X_j^p$. Hence
$$
\Int(\sigma_1)\cap U=\varnothing.
$$
Consequently,
$$
\Int(\gamma_1)\cap \Int(\sigma_1)=\varnothing.
$$

Moreover, the common arc $[p_i,p_j]^+_\gamma$ lies on the common boundary of
$\gamma$ and $\gamma_1$, and the interiors of $\gamma$ and $\gamma_1$ are locally
on the same side of this arc. Hence
$$
[p_i,p_j]^+_\gamma
\subset
\overline{\Int(\gamma)\cap \Int(\gamma_1)}.
$$
Similarly,
$$
[q_i,q_j]^+_\sigma
\subset
\overline{\Int(\sigma)\cap \Int(\sigma_1)}.
$$
By Lemma~\ref{InOut}, it follows that
$$
o(\gamma,\gamma_1)=1,
\qquad
o(\sigma,\sigma_1)=1.
$$
Since $\gamma$ and $\sigma$ are both oriented counterclockwise, we also have
$$
o(\gamma,\sigma)=1.
$$
Therefore
$$
o(\gamma_1,\sigma_1)=1.
$$

Now let $\alpha$ be a nontrivial common initial subarc of
$$
[p,p_i]_{X_i^p}
\quad\text{and}\quad
[p,q_i]_{X_i^p}.
$$
The curves $\gamma_1$ and $\sigma_1$ induce the same orientation on $\alpha$.
However, since
$$
\Int(\gamma_1)\cap\Int(\sigma_1)=\varnothing,
$$
we have
$$
\alpha\cap \overline{\Int(\gamma_1)\cap\Int(\sigma_1)}=\varnothing.
$$
Applying Lemma~\ref{InOut} again gives
$$
o(\gamma_1,\sigma_1)=-1.
$$
This contradicts the previous equality $o(\gamma_1,\sigma_1)=1$.
Therefore our assumption was false and we have
$$
(q_i,q_j)^+_\sigma\subset U.
$$
In particular, the  arc $[q_i,q_j]^+_\sigma$
contains no point $q_k$ with $k\neq i,j$. Hence, by definition,
$$
q_i<_\sigma q_j.
$$
The lemma is proved.
\end{proof}

Let $\gamma\subset \mathbb B_r(p)$ be a definable Jordan curve such that
$p\in \Int(\gamma)$ and each intersection $X_i^p\cap \gamma$ consists of exactly one point,
denoted by $p_i$. Such a curve will be called a \emph{Jordan curve around $p$}.

We define the local cyclic order at $p$ on the set of branches $\{X_1^p, \dots, X_n^p\}$ by stating that $X_i^p < X_j^p$ if and only if $p_i <_\gamma p_j$ for a Jordan curve $\gamma$ around $p$. Lemma \ref{Not-depend} ensures this order is independent of the choice of $\gamma$.

Let $h:X\to \bb R^2$ be a definable continuous injective map, and let $p\in X$ with
$\deg(p,X)=n\ge 3$. We say that $h$ \emph{preserves} the local cyclic order at $p$ if
$$
X_i^p<X_j^p \;\Longrightarrow\; h(X_i^p)<h(X_j^p)
\qquad \text{for all } 1\le i\ne j\le n,
$$
and \emph{reverses} the local cyclic order at $p$ if
$$
X_i^p<X_j^p \;\Longrightarrow\; h(X_j^p)<h(X_i^p)
\qquad \text{for all } 1\le i\ne j\le n.
$$
Accordingly, we define
$$
\ord(h,p)=\left\{
\begin{array}{rl}
 1, & \text{if $h$ preserves the local cyclic order at $p$,}\\
-1, & \text{if $h$ reverses the local cyclic order at $p$,}\\
0, & \text{if otherwise.}
\end{array}\right.
 $$

The following result follows directly from the
definition of local cyclic order.
\begin{lemma}\label{ord-remove}
Let $Y\subset X$ be a definable subset such that $\deg(p,Y)\ge 3$. If $\ord(h,p)\in \{\pm 1\}$,
then
$$
\ord(h|_Y,p)=\ord(h,p).
$$
\end{lemma}

Let $\gamma\in \mathfrak J(X)$ be an oriented definable Jordan curve, let
$p\in \gamma$, and let $\alpha$ be a definable arc with endpoint $p$ such that
$\alpha\cap \gamma=\{p\}$. Let $\gamma^p_1$ and $\gamma^p_2$ be the two local
branches of $\gamma$ at $p$, where $\gamma^p_1$ is the incoming branch and
$\gamma^p_2$ is the outgoing branch with respect to the given orientation of
$\gamma$. Let $\alpha^p$ denote the local branch of $\alpha$ at $p$.

\begin{lemma}\label{OCO}
The curve $\gamma$ is counterclockwise if and only if one of the following two
conditions holds:
\begin{enumerate}
\item[(i)] if $\alpha\setminus\{p\}\subset \operatorname{Ext}(\gamma)$, then
$$
\gamma^p_1<\alpha^p<\gamma^p_2<\gamma^p_1;
$$
\item[(ii)] if $\alpha\setminus\{p\}\subset \Int(\gamma)$, then
$$
\gamma^p_1<\gamma^p_2<\alpha^p<\gamma^p_1.
$$
\end{enumerate}
\end{lemma}

\begin{proof}
Choose $r>0$ sufficiently small so that the circle
$$
C:=S_r(p)
$$
meets each of the three local branches $\gamma^p_1,\gamma^p_2,\alpha^p$ in exactly
one point. Denote these points by
$$
p_1=C\cap \gamma^p_1,\qquad
p_2=C\cap \gamma^p_2,\qquad
q=C\cap \alpha^p.
$$
We orient $C$ counterclockwise.

Since $\gamma^p_1$ is the incoming branch and $\gamma^p_2$ is the outgoing
branch, the curve $\gamma$ passes through $p$ locally from $p_1$ to $p_2$.
If $\gamma$ is counterclockwise, then the interior of $\gamma$ lies locally on
the left-hand side of this oriented passage through $p$. Therefore, on the small
circle $C$, the points of $C$ lying in $\Int(\gamma)$ are exactly
the points of the open counterclockwise arc
$$
(p_2,p_1)^+_C.
$$
Equivalently,
$$
C\cap \Int(\gamma)=(p_2,p_1)^+_C
\quad \text{ and} \quad
C\cap \operatorname{Ext}(\gamma)=(p_1,p_2)^+_C.
$$

If $\alpha\setminus\{p\}\subset \Ext(\gamma)$, then $q\in (p_1,p_2)^+_C$. Therefore
$$
p_1<_C q<_C p_2<_C p_1,
$$
which means
$$
\gamma_1^p<\alpha^p<\gamma_2^p<\gamma_1^p.
$$

If $\alpha\setminus\{p\}\subset \Int(\gamma)$, then $q\in (p_2,p_1)^+_C$, hence
$$
p_1<_C p_2<_C q<_C p_1,
$$
which means
$$
\gamma_1^p<\gamma_2^p<\alpha^p<\gamma_1^p.
$$

Conversely, if one of the two cyclic orders in (i) or (ii) holds, then the point $q$ lies respectively
on the exterior arc $(p_1,p_2)^+_C$ or on the interior arc $(p_2,p_1)^+_C$. This determines which
side of $\gamma$ is the interior near $p$, and therefore shows that $\gamma$ is counterclockwise.
\end{proof}

\begin{remark}\label{OCOR}\rm
\begin{enumerate}
\item[(a)] From the proof of Lemma~\ref{OCO}, it is clear that the induced orientation on
$\gamma$ does not depend on the choice of the point $p\in \gamma$ or of the auxiliary arc
$\alpha$.
\item[(b)] Similarly, one can prove that $\gamma$ is clockwise if and only if one of the following
holds:
\begin{itemize}
\item $\alpha\cap \Int(\gamma)=\varnothing$ and the cyclic order at $p$ in $\gamma\cup \alpha$ is
$$
\gamma_1^p<\gamma_2^p<\alpha^p<\gamma_1^p;
$$
\item $\alpha\setminus \{p\}\subset \Int(\gamma)$ and the cyclic order at $p$ in $\gamma\cup \alpha$ is
$$
\gamma_1^p<\alpha^p<\gamma_2^p<\gamma_1^p.
$$
\end{itemize}
\end{enumerate}
\end{remark}

The next lemma relates the local cyclic order at a point to the orientation of a Jordan
curve passing through that point.

\begin{lemma}\label{ord-add1}
Let $\gamma\in \ak J(X)$, let $p\in \gamma$, and let $\alpha\subset X$ be a definable arc such that
$\alpha\cap \gamma=\{p\}$. Assume that $\ord(h,p)\in\{\pm 1\}$. Then
$$
\ord(h,p)=w(h,\gamma)
$$
if and only if one of the following holds:
\begin{enumerate}
\item $\alpha\cap \Int(\gamma)=\varnothing$ and $h(\alpha)\cap \Int(h(\gamma))=\varnothing$;
\item $\alpha\setminus \{p\}\subset \Int(\gamma)$ and
      $h(\alpha)\setminus \{h(p)\}\subset \Int(h(\gamma))$.
\end{enumerate}
\end{lemma}

\begin{proof}
By Lemma~\ref{ord-remove},
$$
\ord(h|_{\gamma\cup \alpha},p)=\ord(h,p),
$$
so we may assume $X=\gamma\cup \alpha$.

Assume that $\gamma$ is counterclockwise. We consider the case
$\alpha\cap \Int(\gamma)=\varnothing$; the other case is analogous. By Lemma~\ref{OCO}, the
cyclic order at $p$ is
$$
\gamma_1^p<\alpha^p<\gamma_2^p<\gamma_1^p.
$$

If $\ord(h,p)=1$, then the cyclic order at $h(p)$ on $h(\gamma\cup \alpha)$ is
$$
h(\gamma_1^p)<h(\alpha^p)<h(\gamma_2^p)<h(\gamma_1^p).
$$
Hence, by Lemma~\ref{OCO},
$$
w(h,\gamma)=1
\quad \Longleftrightarrow \quad
h(\gamma)\text{ is counterclockwise}
\quad \Longleftrightarrow \quad
h(\alpha)\cap \Int(h(\gamma))=\varnothing.
$$
Thus $\ord(h,p)=w(h,\gamma)$ if and only if
$h(\alpha)\cap \Int(h(\gamma))=\varnothing$.

If $\ord(h,p)=-1$, then the cyclic order at $h(p)$ becomes
$$
h(\gamma_1^p)<h(\gamma_2^p)<h(\alpha^p)<h(\gamma_1^p).
$$
By Remark~\ref{OCOR}(b),
$$
w(h,\gamma)=-1
\quad \Longleftrightarrow \quad
h(\gamma)\text{ is clockwise}
\quad \Longleftrightarrow \quad
h(\alpha)\cap \Int(h(\gamma))=\varnothing.
$$
Again we obtain $\ord(h,p)=w(h,\gamma)$ if and only if
$h(\alpha)\cap \Int(h(\gamma))=\varnothing$.
This proves the lemma.
\end{proof}

\begin{lemma}\label{lem_theta}
Let $\alpha,\beta,\delta$ be pairwise internally disjoint definable arcs with the same endpoints
$p,q$. Assume that
$$
\gamma=\alpha\cup\beta
$$
is the outer Jordan curve, i.e., $\mathring{\delta}\subset \Int(\gamma)$, so
$$
\gamma_1=\alpha\cup\delta,\qquad
\gamma_2=\delta\cup\beta
$$
are the two Jordan curves contained in $\overline{\Int(\gamma)}$.
Equivalently,
$$
\Int(\gamma)
=
\Int(\gamma_1)
\cup \mathring{\delta}
\cup \Int(\gamma_2).
$$

Let $h:\alpha\cup\beta\cup\delta\to \mathbb R^2$ be a definable embedding. Suppose that for some $c\in\{\pm1\}$,
$$
\ord(h,p)=c,\qquad
\ord(h,q)=c,
$$
and
$$
w(h,\gamma_1)=c,\qquad w(h,\gamma_2)=c.
$$
Then
$$
w(h,\gamma)=c.
$$
\end{lemma}

\begin{proof}
Orient all three curves $\gamma,\gamma_1,\gamma_2$ counterclockwise.

The curves $\gamma_1$ and $\gamma_2$ share the arc $\delta$. Since
$\gamma_1$ and $\gamma_2$ are the two smaller Jordan curves inside
$\gamma$, their interiors lie on opposite sides of $\delta$. Hence, by
Lemma 1.2, the orientations induced by $\gamma_1$ and $\gamma_2$ on
$\delta$ are opposite.
Therefore, the orientations induced by $h(\gamma_1)$ and $h(\gamma_2)$ on
$h(\delta)$ are also opposite.

Since
$$
w(h,\gamma_1)=w(h,\gamma_2)=c,
$$
the map $h$ either preserves both orientations or reverses both orientations,  so 
$$o(h(\gamma_1),h(\gamma_2))=o(\gamma_1,\gamma_2)=1.$$

Again by Lemma 1.2, the interiors of
$h(\gamma_1)$ and $h(\gamma_2)$ lie on opposite sides of $h(\delta)$.
Thus the two image curves $h(\gamma_1)$ and $h(\gamma_2)$ bound the two
small regions separated by the common arc $h(\delta)$.
Consequently,
$$
h(\mathring{\delta})\subset \Int(h(\gamma)).$$
This, together with the assumption $\mathring{\delta}\subset \Int(\gamma)$ and Lemma 1.7, gives
$$
\ord(h,p)=w(h,\gamma).
$$
By assumption,
$
\ord(h,p)=c.
$
Therefore
$
w(h,\gamma)=c.
$
\end{proof}

Given two Jordan curves  $\gamma_1,\gamma_2$,
we write $\gamma_1\leq \gamma_2$,
if
$$
\Int(\gamma_1)\subset \Int(\gamma_2);
$$
and, write $\gamma_1<\gamma_2$ if and only if $\gamma_1 \leq \gamma_2$ and $\gamma_1\neq \gamma_2$.

\begin{lemma}\label{BS}
Let $\gamma_1\neq \gamma_2$ be Jordan curves. 
Then
$$
\gamma_1\subset \overline{\Int(\gamma_2)}
\quad\Longleftrightarrow\quad
\gamma_1<\gamma_2.
$$
\end{lemma}

\begin{proof}
If $\gamma_1<\gamma_2$, then $\Int(\gamma_1)\subset \Int(\gamma_2)$ and hence
$$
\gamma_1\subset \overline{\Int(\gamma_1)}\subset \overline{\Int(\gamma_2)}.
$$

Conversely, assume that $\gamma_1\subset \overline{\Int(\gamma_2)}$ and $\gamma_1\neq \gamma_2$.

Since $\gamma_1 \subset \overline{\Int(\gamma_2)}$, we have
$$
\Ext(\gamma_2)= \mathbb R^2 \setminus \overline{\Int(\gamma_2)}
\subset \mathbb R^2 \setminus \gamma_1
= \Int(\gamma_1) \sqcup \Ext(\gamma_1).
$$
As $\Ext(\gamma_2)$ is connected, either $\Ext(\gamma_2) \subset \Int(\gamma_1)$ or $\Ext(\gamma_2) \subset \Ext(\gamma_1)$.
The first possibility is impossible, since $\Ext(\gamma_2)$ is unbounded whereas $\Int(\gamma_1)$ is bounded.
Therefore $\Ext(\gamma_2) \subset \Ext(\gamma_1)$, which implies
$$
\overline{\Int(\gamma_1)} \subset \overline{\Int(\gamma_2)}.
$$
It follows that $\Int(\gamma_1) \subset \Int(\gamma_2)$, i.e.\ $\gamma_1\leq \gamma_2$. But $\gamma_1 \neq \gamma_2$, hence $\gamma_1 < \gamma_2$.
\end{proof}

\subsection{The Definable Schoenflies Theorem}
\begin{theorem}[Definable Schoenflies theorem]\label{Thm_definable_Schoenflies}
Let $\gamma \subset \mathbb R^2$ be a definable Jordan curve, and let 
$h:\gamma \to \mathbb R^2$ be a definable embedding. 
Then there exists a definable homeomorphism 
$H:\mathbb R^2 \to \mathbb R^2$ such that
$$
H|_\gamma = h .
$$
\end{theorem}

\begin{proof}
Set $\gamma' := h(\gamma)$. Since $h$ is a definable embedding, 
$\gamma'$ is again a definable Jordan curve.

By the definable triangulation theorem, there exist polygonal Jordan curves 
$P$ and $P'$ in $\mathbb R^2$ and definable homeomorphisms
$$
H_1, H_2:\mathbb R^2 \to \mathbb R^2
$$
such that
$$
H_1(P)=\gamma
\quad \text{and} \quad
H_2(P')=\gamma'.
$$
Consider the composition
$$
h' := H_2^{-1}\circ h\circ H_1 : P \to P' .
$$
Then $h'$ is a definable homeomorphism between polygonal Jordan curves.

By similar arguments as in the proof of  polygonal Schoenflies theorem 
\cite[Theorem III.1.C]{Bing1983}, we obtain a definable homeomorphism
$$
H' : \mathbb R^2 \to \mathbb R^2
$$
such that
$$
H'|_P = h' .
$$
Therefore, 
$$
H := H_2\circ H'\circ H_1^{-1}
$$
is a desired definable extension of $h$.
\end{proof}

\section{Topological plane graphs}\label{section2}
\begin{definition}\rm 
A \emph{definable plane graph} (or simply a \emph{plane graph}) is a pair
$G=(V,E)$ embedded in $\mathbb R^2$ satisfying the following conditions:
\begin{enumerate}[label=(\alph*)]
    \item $V\subset \mathbb{R}^2$ is a finite set, whose elements are called \emph{vertices};
    \item each \emph{edge} $e\in E$ is either
    \begin{itemize}
        \item a definable arc whose endpoints are two  distinct vertices of $V$, or
        \item a definable loop with endpoint in $V$;
    \end{itemize}
    \item edges  intersect only at common endpoints, that is,
    for any two distinct edges $e,e'\in E$,
    $$
    \mathring e \cap \mathring e' = \varnothing;
    $$
    \item apart from its endpoints, an edge contains no vertices.
\end{enumerate}  
\end{definition}

\begin{remark}\rm
By a slight abuse of notation, when no confusion arises, we identify the
plane graph $G=(V,E)$ with the subset of $\mathbb R^2$ given by the union
of its vertices and edges.
A Jordan curve in $G$ is precisely a circuit in $G$. We denote by
$\mathfrak J(G)$ the set of all Jordan curves contained in $G$.
\end{remark}

The connected components of $\mathbb{R}^2 \setminus G$ are called the \emph{faces} of $G$.
There is a unique unbounded face, called the \emph{outer face}; all others are called
\emph{inner faces}. We denote  $\ak F_\infty (G)$ the outer face of $G$.

\begin{lemma}\label{lem3.1}
Let $G=(V,E)$ be a plane graph and let $\gamma\in \ak J(G)$. Let $e\in E$ be an arbitrary edge. Then
\begin{equation}\label{equ3.1}
\gamma \cap \mathring e \neq \varnothing
\quad \Longrightarrow \quad
\mathring e \subset \gamma.
\end{equation}
Consequently,  $\gamma$ is a finite union of edges of $G$.
\end{lemma}

\begin{proof}
Indeed, suppose that $\gamma\cap \mathring e\neq\varnothing$ and set
$\alpha:=\gamma\cap \mathring e$.
Since $\mathring e$ contains no vertices and does not intersect the interior of any other
edge, each point of $\mathring e$ has a neighborhood in $G$ homeomorphic to an open
interval.
Hence,  for any $x\in \alpha$, a small neighborhood of $x$ in $\gamma$ coincides with
a subarc of $\mathring e$, which shows that $\alpha$ is open in $\mathring e$.

On the other hand, since $\gamma$ is a closed subset of $\mathbb{R}^2$, the set $\alpha$ is
closed in $\mathring e$.
As $\mathring e$ is connected, it follows that $\alpha=\mathring e$.

Therefore, for each edge $e\in E$, either $\gamma\cap \mathring e=\varnothing$ or
$\mathring e\subset \gamma$.
Since $\gamma\subset G$ and $G$ has finitely many edges, $\gamma$ is the union of finitely
many edges of $G$.
\end{proof}

Let $G=(V,E)$ be a plane graph. A graph $G'=(V',E')$ is called a \emph{subgraph} of $G$ if $V'\subset V$ and $E'\subset E$. If $V'=V$ and $G'$ contains no Jordan curves, then $G'$ is called a
\emph{spanning forest} of $G$.
When $G$ is connected, any connected spanning forest of $G$ is called a
\emph{spanning tree}.

A vertex $v\in V$ is called a \emph{leaf} if it is incident to exactly one edge
which is not a loop.
By convention, an isolated vertex is also called a leaf.

\begin{lemma}\label{graph-properties}
Let $G$ be a plane graph, $e$ be an edge of $G$, and $F$ be a face of $G$.
Then the following statements hold:

(i) Either $e\subset \partial F$, or $\mathring e\cap \partial F=\varnothing$, where $\partial F$ denotes the boundary of $F$.

(ii) If $e$ lies on a Jordan curve in $G$, then $e$ is contained in the boundary of exactly two faces of $G$.

(iii) If $e$ lies on no Jordan curve, then $e$ is contained in the boundary of exactly one face of $G$.

(iv) If $G'$ is a subgraph of $G$, then every face of $G$ is contained in a face of $G'$.
In particular, the outer face of $G$ is contained in the outer face of $G'$.

(v) Let $G$ and $G'$ be disjoint plane graphs, and let $e$ be an arc joining $G$ to $G'$ such that
$\mathring e\subset \bb R^2\setminus(G\cup G')$.
Then
$$
\ak  J(G\cup G')=\ak  J(G\cup G'\cup e).
$$

(vi) If $G$ is connected and a spanning tree of $G$ consists of a single vertex $p$, then $G$  is either the point $p$ or a union of loops based at $p$.

(vii) Let $\gamma_1,\gamma_2\in \ak  J(G)$ with $\gamma_1<\gamma_2$, and assume that there exists an edge
$e\subset \gamma_1\cap \gamma_2$.
Let $F$ be the face of $G$ such that $e\subset \partial F$ and $F\subset \Int(\gamma_2)$.
Then $F\subset \Int(\gamma_1)$.

(viii) If $F$ is an inner face of $G$, then there exists a unique $\gamma\in \ak  J(\partial F)$ such that
$F\subset \Int(\gamma)$. In particular, $\gamma= \partial F \cap \partial(\ak F_\infty(\partial F))$.

(ix) For any two distinct faces of $G$, there exists a Jordan curve $\gamma\in \ak  J(G)$ separating them:
one face lies in $\Int(\gamma)$, while the other lies in $\Ext(\gamma)$.
Moreover, $\gamma$ may be chosen to lie in the boundary of the face contained in $\Int(\gamma)$.

(x) A point of $G$ is not on the boundary of the outer face of $G$ if and only if it lies in the interior of some Jordan curve in $G$.

\end{lemma}

\begin{proof}
Items (i)--(iii) follow from~\cite[Lemma~4.2.2]{Diestel2018}, and  (iv) from~\cite[Lemma~4.2.1]{Diestel2018}.

\medskip\noindent
(v)
The inclusion $\ak  J(G\cup G')\subset \ak  J(G\cup G'\cup e)$ is immediate.
For the reverse inclusion, let $\gamma\in \ak  J(G\cup G'\cup e)$.
If $e\subset \gamma$, then $\gamma\setminus \mathring e\subset G\cup G'$ is a path joining an endpoint in $G$ to an endpoint in $G'$.
This contradicts the assumption that $G$ and $G'$ are disjoint.
Hence $e\not\subset \gamma$, so $\gamma\subset G\cup G'$ and $\gamma\in \ak  J(G\cup G')$.

\medskip\noindent
(vi) If $G$ has no edges, then $G =  \{p\}$. We may assume that $G$ has edges. 
Since a spanning tree of $G$ has only one vertex $p$,  every edge of $G$ has both endpoints equal to $p$.
Thus, every edge is a loop based at $p$, and $G$ is a union of such loops.

\medskip\noindent
(vii)
Pick a point $x\in \mathring e$ and a disk $D$ centered at $x$ such that
$D\cap \gamma_1 = D\cap \gamma_2$ is an open arc contained in $\mathring e$.
Then $D\setminus \gamma_1=D\setminus \gamma_2$ and this set has exactly two connected components, one contained in
$\Int(\gamma_i)$ and the other contained in $\Ext(\gamma_i)$ for $i=1,2$.
We have 
$$D \cap \Int(\gamma_1) \subset D \cap \Int(\gamma_2)\subset F,$$
where the first inclusion follows from the assumption $\gamma_1<\gamma_2$; while the second one is a consequence of the assumptions $e\subset \partial F$ and $F\subset \Int(\gamma_2)$.
Thus $F\cap\Int(\gamma_1)\ne\varnothing.$
On the other hand, since $F$ is connected and disjoint from $\gamma_1$, it follows that $F \subset \Int(\gamma_1)$.

\medskip\noindent

(viii) 
Note that $\partial F$ is also a graph with $F$ as an inner face.
Let $O=\mathfrak F_\infty(\partial F)$--the outer face of $\partial F$,
let $e$ be a common edge of $\partial F$ and $\partial O$, let
$p\in \mathring e$, and let $D$ be a small disk centered at $p$ such that
$D\cap \mathring e$ is an open arc. Then $D\setminus e$ has exactly two
connected components, say $D_1$ and $D_2$, one of which is contained in
$F$ and the other in $O$. Without loss of generality, we may assume that
$D_1\subset F$ and $D_2\subset O$.
By (iii), there exists a Jordan curve
$\gamma\in \mathfrak J(\partial F)$ containing $e$. Since
$D\setminus e=D\setminus \gamma$, the sets $D_1$ and $D_2$ lie on
different sides of $\gamma$. We may regard $\gamma$ as a subgraph of
$\partial F$, so by (iv), $O\subset \Ext(\gamma)$. Since
$D_2\subset O$, it follows that $D_2\subset \Ext(\gamma)$. Hence
$D_1\subset \Int(\gamma)$, and therefore $F\subset \Int(\gamma)$.
We now show that $\gamma$ is unique.

Indeed, if $\sigma \in \frak J(\partial F)$ is such that $F \subset \Int(\sigma)$, then we have $F \subset \Int(\gamma) \cap \Int(\sigma)$. It follows that $\partial F \subset \overline{\Int(\gamma) \cap \Int(\sigma)}$, which implies $\gamma \cup \sigma \subset \overline{\Int(\gamma) \cap \Int(\sigma)}$. In particular, this yields
$$\gamma \subset \overline{\Int(\sigma)} \quad \text{and} \quad \sigma \subset \overline{\Int(\gamma)}.$$
Therefore, $\sigma = \gamma$.
By the construction, every common edge of $\partial F$ and $\partial O$
is contained in $\gamma$. Hence $\gamma = \partial O \cap \partial F$.

\medskip\noindent
(ix)
If one of the faces is the outer face $\mathfrak F_\infty(G)$, let $F$ be the other (inner) face.
By (viii), there exists $\gamma\in \ak  J(\partial F)\subset \ak  J(G)$ with $F\subset \Int(\gamma)$, so $\mathfrak F_\infty(G)$ lies in the unbounded component of $\bb R^2\setminus \gamma$ and $\gamma$ separates the two faces.

Now assume that both $F_1$ and $F_2$ are inner faces.
Pick $\gamma_i\in \mathfrak J(\partial F_i)$ such that
$F_i\subset \Int(\gamma_i)$ for $i=1,2$ (in view of Item~(viii)). If
$\Int(\gamma_1)\cap F_2=\varnothing$, then $\gamma:=\gamma_1$
separates $F_1$ and $F_2$.
Otherwise, $F_2\subset \Int(\gamma_1)$. Hence
$\gamma_2\subset \partial F_2\subset \overline{\Int(\gamma_1)}$, and
therefore
$\Int(\gamma_2)\subset \Int(\gamma_1)$.
If $\gamma_1=\gamma_2$, then both $F_1$ and $F_2$ are inner faces bounded
by the same Jordan curve. Let $e\subset \gamma_1$ be an edge of $G$. Then
clearly $e\subset \partial F_1$ and $e\subset \partial F_2$. On the other
hand, since $\gamma_1$ is a Jordan curve, $e$ lies on the boundary of
exactly two faces: one contained in $\Int(\gamma_1)$ and the other in
$\Ext(\gamma_1)$. Since both $F_1$ and $F_2$ are contained in
$\Int(\gamma_1)$, we must have $F_1=F_2$, a contradiction. Therefore,
$\gamma_2<\gamma_1$.
Choose an edge $e\subset \gamma_1$ such that
$\mathring e\subset \Ext(\gamma_2)$. Since $e\subset \partial F_1$, the
face $F_1$ lies in $\Ext(\gamma_2)$, while
$F_2\subset \Int(\gamma_2)$. Hence $\gamma:=\gamma_2$ separates $F_1$
and $F_2$, and $\gamma\subset \partial F_2$.

\medskip\noindent
(x) ($\Rightarrow$)
Assume that $p\in G\setminus \partial \mathfrak F_\infty(G)$.
Let $e_1,\dots,e_n$ be the edges of $G$ containing $p$ (set $n=0$ if no
edge contains $p$).
By item~(i), we have
$$
\mathring e_i\cap \partial \mathfrak F_\infty(G)=\varnothing
\qquad \text{for all } i=1,\dots,n.
$$
Remove the set
$$
\{p\}\cup \bigcup_{i=1}^n \mathring e_i
$$
from $G$, and denote the resulting graph by $G'$.
Then $G'$ has the same outer face as $G$, and the point $p$ lies in an
inner face of $G'$, say $F'$.
By item~(viii), there exists
$\gamma\in \mathfrak J(G')\subset \mathfrak J(G)$ such that
$F'\subset \Int(\gamma)$.
In particular, $p\in \Int(\gamma)$.

($\Leftarrow$)
Conversely, if $q\in \Int(\gamma)$ for some $\gamma\in \ak  J(G)$, then every face whose boundary contains $q$ must be contained in $\Int(\gamma)$.
In particular, $q$ cannot lie on the boundary of the outer face $\mathfrak F_\infty(G)$.
\end{proof}

Let us denote by $\Int(G)$ the union of all inner faces of $G$ and set
$$
B(G) := G \cup \Int(G).
$$

\begin{remark}\rm 
Observe that $B(G)$ is the complement of the outer face of $G$. Since the outer face is open, it follows that $B(G)$ is a closed subset of
$\mathbb R^2$.
\end{remark}

\begin{lemma}\label{lem_BG-boundary}
We have $\partial B(G) \subset G$.
\end{lemma}

\begin{proof}

By definition $B(G) = G \cup \Int(G)$ which is closed subset of $\bb R^2$.
Thus any point of $\partial B(G)$ belongs either to $G$ or to
$\Int(G)$.
However, every point of $\Int(G)$ is a topological interior point of $B(G)$,
and hence cannot lie in $\partial B(G)$.
Therefore $\partial B(G) \subset G$.
\end{proof}

By Lemma \ref{lem_BG-boundary} we have $\partial B(G) \subset G$.  We now endow $\partial B(G)$ with a plane graph structure as follows.
First we specify the vertices:
\begin{itemize}
    \item Every topological singular point of $\partial B(G)$ (i.e.\ a point at which
    $\partial B(G)$ is not locally homeomorphic to an open interval) is defined to be a vertex of
    $\partial B(G)$.
    \item If a connected component of $\partial B(G)$ is a Jordan curve, we choose a vertex
    of $G$ lying on this component and define it to be a vertex of $\partial B(G)$.
\end{itemize}

Next we specify the edges:
\begin{itemize}
    \item An arc in $\partial B(G)$ whose endpoints are vertices of $\partial B(G)$ is
    defined to be an edge of $\partial B(G)$ if its relative interior contains no other vertices
    of $\partial B(G)$.
    \item A Jordan curve contained in $\partial B(G)$ is defined to be an edge of
    $\partial B(G)$ if it contains exactly one vertex of $\partial B(G)$; in this case
    it is a loop.
\end{itemize}
In this way, $\partial B(G)$ becomes a plane graph in the sense defined above.
Note, however, that with this graph structure, $\partial B(G)$ is not, in general, a subgraph of $G$.

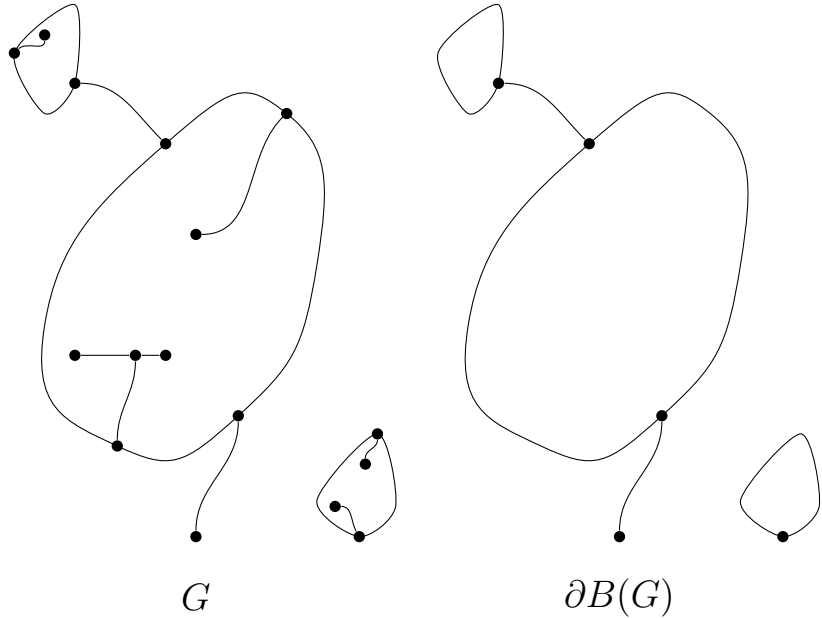
\begin{figure}[htbp]
\centering
\begin{tikzpicture}[scale=0.8, every node/.style={circle, fill=black, inner sep=1.5pt}]

    \begin{scope}[shift={(0,0)}]
        \node (p1) at (0,3) {};
        \node (p2) at (2,3.5) {};
        \node (p3) at (1.2,-1.5) {};
        \node (p4) at (-0.8,-2) {};
        
        \draw [smooth cycle, tension=1] plot coordinates {(p1) (2,3.5) (2.5,1) (p3) (p4) (-2,0)};
        
        \node (tl1) at (-1.5,4) {};
        \node (tl2) at (-2.5,4.5) {};
        \node (tl_pendant) at (-2,4.8) {};
        
        \draw (p1) to [out=135, in=0] (tl1);
        \draw [smooth cycle] plot coordinates {(tl1) (-2,3.5) (tl2) (-1.5,5.3)};
        \draw (tl2) to [out=45, in=270] (tl_pendant);
        
        \node (pendant1) at (0.5,1.5) {};
        \draw (2,3.5) to [out=220, in=0] (pendant1);
        
        \node (loop_node1) at (0,-0.5) {};
        \node (loop_node2) at (-0.5,-0.5) {};
        \draw (loop_node1) -- (loop_node2);
        
        \node (pendant2) at (-1.5,-0.5) {};
        \draw (loop_node2) -- (pendant2);
        \draw (loop_node2) to [out=270, in=90] (p4);
        
        \node (pendant3) at (0.5,-3.5) {};
        \draw (p3) to [out=270, in=90] (pendant3);
        
        \node (br1) at (3.5,-1.8) {};
        \node (br2) at (3.2,-3.5) {};
        \node (br_pendant1) at (3.3,-2.3) {};
        \node (br_pendant2) at (2.8,-3.0) {};
        
        \draw [smooth cycle] plot coordinates {(br1) (2.5,-2.9) (br2) (3.8,-3)};
        \draw (br1) to [out=270, in=90] (br_pendant1);
        \draw (br2) to [out=120, in=0] (br_pendant2);
        
        \node[fill=none] at (0.5,-4.5) {\large $G$};
    \end{scope}

    \begin{scope}[xshift=7cm]
        \node (bp1) at (0,3) {};
        \node (bp3) at (1.2,-1.5) {};
        
        \draw [smooth cycle, tension=1] plot coordinates {(bp1) (2,3.5) (2.5,1) (bp3) (-0.8,-2) (-2,0)};
        
        \node (btl1) at (-1.5,4) {};
        \draw (bp1) to [out=135, in=0] (btl1);
        \draw [smooth cycle] plot coordinates {(btl1) (-2,3.5) (-2.5,4.5) (-1.5,5.3)};
        
        \node (bpendant3) at (0.5,-3.5) {};
        \draw (bp3) to [out=270, in=90] (bpendant3);
        
        \node (bbr2) at (3.2,-3.5) {};
        \draw [smooth cycle] plot coordinates {(3.5,-1.8) (2.5,-2.9) (bbr2) (3.8,-3)};
        
        \node[fill=none] at (0.5,-4.5) {\large $\partial B(G)$};
    \end{scope}

\end{tikzpicture}
\caption{$G$ and $\partial B(G)$}
\label{Fig1}
\end{figure}

\begin{remark}\label{rem:deg2-loop}\rm
A vertex of $\partial B(G)$ has the degree $2$ if and only if the connected component of
$\partial B(G)$ containing this vertex is a loop.
\end{remark}

\begin{proposition}\label{prop:boundary-properties}
Let $G$ be a plane graph and consider $\partial B(G)$ endowed with the plane graph
structure defined above. Then the following hold:

(i) $V(\partial B(G)) \subset V(G)$.
    
(ii) For every Jordan curve $\gamma \in \ak J(\partial B(G))$, we have
    $$
    \Int(\gamma)\cap \partial B(G)=\varnothing.
    $$
    
(iii) Assume that $\partial B(G)$ is connected and contains at least one edge.
    Then every leaf of any spanning tree of $\partial B(G)$ is either a leaf of
    $\partial B(G)$ or the endpoint of a loop of $\partial B(G)$.
\end{proposition}

\begin{proof}
(i) By Lemma \ref{lem_BG-boundary}, we have $\partial B(G)\subset G$.
Let $p\in V(\partial B(G))$.
By construction, either $p$ is a topological singular point of $\partial B(G)$,
or $p$ is a vertex of $G$ chosen on a connected component of $\partial B(G)$
which is a Jordan curve.

In the first case, $p$ is also a topological singular point of $G$.
Hence $p\in V(G)$.
In the second case, $p\in V(G)$ by definition.
Therefore $V(\partial B(G))\subset V(G)$.

(ii) Let $\gamma\in\ak J(\partial B(G))$.
Since $\gamma\subset\partial B(G)\subset G$, the interior $\Int(\gamma)$
is disjoint from the outer face of $G$, and hence
$\Int(\gamma)\subset B(G)$. Every point of $\Int(\gamma)$ is a topological interior point of $B(G)$, and
in particular does not belong to $\partial B(G)$.
Hence
$$
\Int(\gamma)\cap\partial B(G)=\varnothing.
$$

(iii)
Assume that $\partial B(G)$ is connected and contains at least one edge, and let $T$ be a spanning tree of
$\partial B(G)$. Let $p$ be a leaf of $T$.

If $p$ is a leaf of $\partial B(G)$, then there is nothing to prove.
Assume that $p$ is not a leaf of $\partial B(G)$ and that $p$ is not the endpoint of a loop of $\partial B(G)$.
Then $\deg(p, \partial B(G))\ge 3$ (indeed, $\deg(p, \partial B(G))=2$ would imply, by
Remark~\ref{rem:deg2-loop}, that the connected component of $\partial B(G)$ containing $p$ is a loop).

Let $e$ be the unique edge of $T$ incident to $p$.
Since $\deg(p, \partial B(G))\geq 3$ and $p$ is not the endpoint of a loop of $\partial B(G)$, there exist two distinct edges
$e_1,e_2\in E(\partial B(G))\setminus E(T)$ incident to $p$.
Let $u_1$ and $u_2$ be the other endpoints of $e_1$ and $e_2$, respectively.

For $i=1,2$, let $P_i$ be the unique path in $T$ joining $p$ and $u_i$.
Let $w$ be the first vertex (starting from $p$) at which the paths $P_1$ and $P_2$ diverge; equivalently,
$P_1$ and $P_2$ share an initial subpath $P_0$ from $p$ to $w$, and after $w$ they are edge-disjoint.
Write $P_i=P_0\cup Q_i$, where $Q_i$ is the subpath from $w$ to $u_i$.

Now consider the following three paths joining $p$ and $w$:
$$
P_0,\qquad
R_1:=e_1\cup Q_1,\qquad
R_2:=e_2\cup Q_2.
$$
These paths are internally disjoint (they meet only at the endpoints $p$ and $w$), and their union
$$
H:=P_0\cup R_1\cup R_2
$$
is a subgraph of $\partial B(G)$.

Let $U$ be the unbounded component of $\bb R^2\setminus H$, and let $\gamma:=\partial U$.
Then $\gamma$ is a Jordan curve contained in $H\subset \partial B(G)$.
Moreover, $\gamma$ is the union of exactly two of the three paths joining $p$ and $w$, so the remaining third path lies in
$\Int(\gamma)$. In particular, $\Int(\gamma)\cap \partial B(G)\neq\varnothing$.
This contradicts~(ii), which asserts that $\Int(\gamma)\cap \partial B(G)=\varnothing$ for every
$\gamma\in \ak J(\partial B(G))$.

Therefore $p$ must be either a leaf of $\partial B(G)$ or the endpoint of a loop of $\partial B(G)$.
\end{proof}
\begin{definition}\label{def_definable-plane-graph}\rm 
Let $X \subset \mathbb{R}^2$ be a compact definable set of dimension~$1$. 
We can (and will) regard $X$ as a plane graph by choosing a finite set of
vertices on each connected component of $X$ as follows:
\begin{itemize}
    \item[(i)] Every topological singular point of $X$ is taken to be a vertex.
    \item[(ii)] Every connected component of $X$ of dimension $0$ is a vertex.
    \item[(iii)] For each connected component of $X$ which is a Jordan curve,
    we choose one point on this component and declare it to be a vertex
    (so this component consists of a single vertex and one loop edge).
\end{itemize}
Once the set of vertices is fixed, the edges of the associated plane graph
are the closures of the connected components of $X \setminus V(X)$.
Note that a non-smooth point of $X$ need not be a vertex if it is not a
topological singular point.
\end{definition}

With these remarks, for the remainder of this part, let
$X \subset \mathbb{R}^2$ be a connected  compact definable set of dimension $1$,
considered as a plane graph.

\begin{definition}\rm  
    A Jordan curve $\gamma \subset X$ is called \emph{special} if the following
conditions are satisfied:
\begin{itemize}
    \item $\gamma$ is a loop of $X \setminus \Int(\gamma)$ (the plane graph structure of $X\setminus \Int(\gamma)$ is defined as Definition \ref{def_definable-plane-graph};
    
    \item there exists a unique point $p \in \gamma$ such that there is an edge of
    $X$ containing $p$ and not contained in $\Int(\gamma)$;
    
    \item there exists $\beta \in \ak J(X)$ such that $\gamma < \beta$
    (equivalently, $\gamma$ is not contained in the boundary of the outer face of $X$).
\end{itemize}
\end{definition}

\begin{example}{\rm
In Figure~\ref{Fig2} below, there are two special Jordan curves,
namely the two red Jordan curves $\gamma_1$ and $\gamma_2$.
\begin{figure}[htbp]
\centering
\begin{tikzpicture}[
    thick, 
    dot/.style={circle, fill=black, inner sep=1.5pt},
    label distance=-1mm
]

    \draw (0,0) ellipse (4cm and 2.5cm);
    
    \node[dot] (top_outer) at (0, 2.5) {};
    \node[dot] (bot_outer) at (1.5, -2.32) {};
    \node[dot] (bridge_left) at (4, 0) {};

    \draw[red] (-1.5,0.2) ellipse (1.2cm and 0.7cm);
    \node[red, below] at (-1.9, -0.5) {$\gamma_1$};
    
    \draw (-1.5,0.2) ellipse (0.5cm and 0.25cm);
    \node[red, above] at (-1.5, -0.5) {};
    
    \node[dot] (top_red) at (-1.5, 0.9) {};
    \node[dot] (tr_red) at (-0.8, 0.6) {};
    \node[dot] (inner_dot) at (-1.2, 0.4) {};
    \node[dot] (left_red) at (-2.65, 0.4) {};
    \node[dot] (pendant_left) at (-2, 0.2) {};
    
    \draw (top_outer) to [out=270, in=90] (top_red);
    
    \draw (tr_red) to [out=240, in=60] (inner_dot);
    
    \draw (left_red) -- (pendant_left);

    \draw[red] (1.5, -0.2) circle (0.6cm) ;
     \node[red, below] at (1.8, -0.8) {$\gamma_2$};
    
   
    \node[dot] (circle_bot) at (1.5, -0.8) {};


    \draw (circle_bot) to [out=270, in=90] (bot_outer);

    \node[dot] (bridge_right) at (5.2, 0.2) {};
    \draw (bridge_left) to [out=20, in=160] (bridge_right);
    
    \draw (6.5, 0.2) ellipse (1.3cm and 0.9cm);
    \node[red, below] at (6.5, -0.7) {};
\end{tikzpicture}
\caption{Special Jordan curves}
\label{Fig2}
\end{figure}
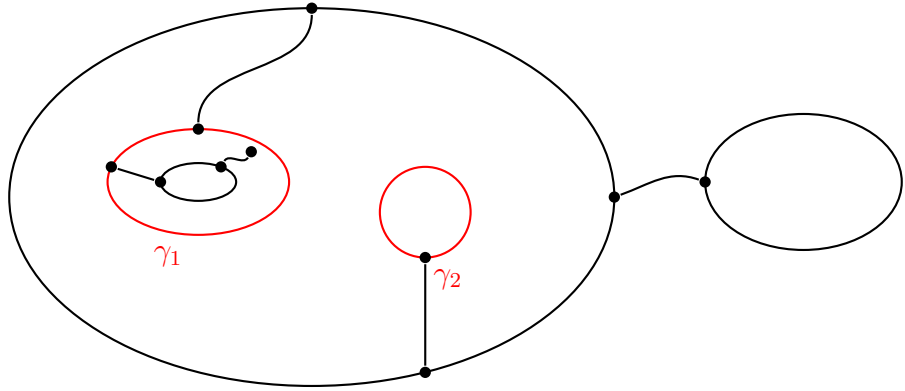}
\end{example}

\begin{lemma}\label{jcncsp}
Let $Y\subset \bb R^2$ be a compact, connected definable set of dimension $1$.
Then $Y$ is a Jordan curve if and only if it has no topological singular points.
\end{lemma}

\begin{proof}
If $Y$ is a Jordan curve, then $Y$ is (locally) a simple arc: every point of $Y$ has a neighborhood in $Y$
homeomorphic to an open interval. Hence $Y$ has no topological singular points.

Conversely, assume that $Y$ has no topological singular points.
By definition, for every $y\in Y$ there exists a neighborhood $U$ of $y$ in $\bb R^2$ such that
$U\cap Y$ is homeomorphic to an open interval.
Therefore $Y$ is a $1$--dimensional topological manifold without boundary.
Since $Y$ is compact and connected, the classification of connected $1$--manifolds
(e.g.\ \cite[Appendix, Theorem]{Milnor}) implies that $Y$ is homeomorphic to the unit circle.
Consequently, $Y$ is a Jordan curve.
\end{proof}

\begin{theorem}\label{uojr}
Let $X\subset \bb R^2$ be a connected, compact, $1$-dimensional definable set, viewed as a plane graph.
If $X$ has no leaves in $\Int(B(X))$ and has no special Jordan curves, then $\Int(X)$ is a disjoint union of Jordan regions.
\end{theorem}

\begin{proof}
Let $U$ be a bounded face of $X$.
By Lemma \ref{jcncsp}, it suffices to show that $\partial U$ has no topological singular points. Indeed, suppose on the contrary that  $p\in \partial U$ is a topological singular point. 
We consider two cases:

\noindent\textbf{Case 1.} $\deg(p,\partial U)=1$.

Let $e$ be the unique edge of $\partial U$ incident to $p$.
Choose a sufficiently small closed disk $D$ centered at $p$ such that $D\cap X$ is a subarc of $e$. Since $U$ is a face and $e\subset \partial U$, and $e$ lies on no Jordan curve, by Lemma \ref{graph-properties} (iii),  $D\setminus e$
is contained in $U$.

We claim that $p$ is incident to no edge of $X$ other than $e$.
Indeed, if there were another edge $e'\subset X$ with $e'\neq e$ and $p\in e'$, then $e'\cap D$
would contain a subarc leaving $p$ and contained in $D\setminus e$. This subarc must be contained in $U$, and this is impossible since $e' \cap U =  \varnothing$. 
Thus,  $e$ is the only edge of $X$ incident to $p$, and hence $\deg(p,X)=1$. This contradicts the fact that $X$ has no leaves.

\noindent\textbf{Case 2:} $\deg(p,\partial U)\ge 3$ for some $p\in \partial U$.

Since $\partial U\subset X$, we have $d:=\deg(p,X)\ge 3$.
Let $U_1,\dots,U_d$ be the connected components of the germ $(\bb R^2\setminus X,p)$.

We claim that the face $U$ contains at least two distinct components $U_i$ and $U_j$. Indeed, if $U$ contained only one component $U_i$, then locally near $p$ the boundary $\partial U$
would agree with $\partial U_i$, which is bounded by exactly two local branches of $X$.
Hence $\deg(p,\partial U)=2$, a contradiction.

\smallskip
Choose two distinct components $U_i,U_j\subset U$.
Let $X_{i1}^p,X_{i2}^p$ (resp.\ $X_{j1}^p,X_{j2}^p$) be the two local branches of $(X\setminus\{p\},p)$
bounding $U_i$ (resp.\ $U_j$).
Up to relabeling, assume that the local cyclic order at $p$ satisfies
$$
X_{i1}^p < X_{i2}^p \le X_{j1}^p < X_{j2}^p \le X_{i1}^p .
$$

Pick points $p_i\in U_i$ and $p_j\in U_j$ close to $p$. Choose
definable arcs $\alpha_i,\alpha_j$ from $p$ to $p_i,p_j$ such that
$$\alpha_i\setminus\{p\}\subset U_i,\quad \alpha_j\setminus\{p\}\subset U_j \quad \text{and} \quad \alpha_i\cap \alpha_j=\{p\}.$$
Since $U$ is connected, we can choose a definable arc $\alpha\subset U$
joining $p_i$ to $p_j$ and disjoint from
$\alpha_i\cup \alpha_j$ except at $p_i,p_j$. Set
$$
\gamma:=\alpha_i\cup \alpha \cup \alpha_j .
$$
Then $\gamma$ is a Jordan curve satisfying
$$
\gamma \setminus \{p\} \subset U
\quad \text{and} \quad
X\cap \Int(\gamma)\neq\varnothing
$$
(by the cyclic order, two of the branches among
$X_{i1}^p,X_{i2}^p,X_{j1}^p,X_{j2}^p$ lie in $\Int(\gamma)$).

Let $\gamma_U$ be the Jordan curve in $\partial U$ such that $U \subset \Int(\gamma_U)$ obtained from Lemma \ref{graph-properties} (viii). It is clear that  $\gamma < \gamma_U$.
Let $Z$ be the closure of a connected component of $X\cap \Int(\gamma)$.
Then $\partial B(Z)\subset X$ and $\partial B(Z)<\gamma$.

\smallskip
First, assume that $\partial B(Z)$ is a Jordan curve. Since
$\partial B(Z)\subset X$, it lies strictly inside $\gamma$, and hence
inside $\gamma_U$. Moreover, $X$ meets $\gamma$ only at $p$, and since $X$ is connected and contains points outside $B(Z)$, there exists an edge of $X$ incident to $p$ that is not contained in
$\Int(\partial B(Z))$. Therefore, $\partial B(Z)$ is a special Jordan
curve, contradicting the hypothesis.

Thus, $\partial B(Z)$ is not a Jordan curve.
Let $T$ be a spanning tree of the plane graph $\partial B(Z)$.

If $T$ has only one vertex, then by Lemma~\ref{graph-properties} (vi), the graph $\partial B(Z)$ is a union of
(at least two) loops with a common endpoint $p$.
In particular, $\partial B(Z)$ contains a loop $\ell$ with $\ell<\gamma < \gamma_U$; such a loop yields a special Jordan curve of $X$,
again a contradiction.

Hence $T$ has at least one edge and therefore has a leaf $x\neq p$.
By Proposition~\ref{prop:boundary-properties} (iii), the point $x$ is
either a leaf of $\partial B(Z)$ or the endpoint of a loop of
$\partial B(Z)$.
In the former case, $x$ is a leaf of $X$, and since
$x\in \Int(\gamma)$, we have $x\in \Int(B(X))$, contradicting the
assumption that $X$ has no leaves in $\Int(B(X))$.
In the latter case, the loop at $x$ is a special Jordan curve of $X$,
which is also excluded by hypothesis.

In all cases we obtain a contradiction.
Therefore, $\partial U$ has no topological singular points of degree $\ge 3$.
This ends the proof of the theorem.
\end{proof}

\section{Topological extension: the bounded case}\label{section3}
Let $X\subset \mathbb R^2$ be a definable compact set of dimension $\leq 1$. We regard $X$ as a plane graph endowed with the graph structure introduced in Definition~\ref{def_definable-plane-graph}. By Lemma~\ref{graph-properties} \textnormal{(viii)}, for each inner face $F$ of $X$ there exists a unique Jordan curve $\gamma_F$ such that
$$
F \subset \Int(\gamma_F)
\qquad\text{and}\qquad
\gamma_F \subset \partial F.
$$
We call such a curve a \emph{facial Jordan curve}, and denote by $\ak J_{facial}(X)$ the set of all facial Jordan curves of $X$.

Given $c \in \{ \pm 1\}$,
consider the following conditions: 

$(E1)$ $\ord(h,p)=c$ for every $p \in X$ with $\deg(p,X)\ge 3$;

$(E2)$ $w(h,\gamma)=c$ for every  $\gamma \in \ak J(X)$;

$(E2_F)$  $w(h,\gamma) = c$ for every $\gamma \in \ak J_{facial} (X)$;

$(E3)$ for every $\gamma \in \ak J(X)$ and every point $x \in X\setminus \gamma$, we have
  $$
    x \in \Int(\gamma) \iff  h(x)\in \Int(h(\gamma));
  $$

($E3_F$) for every $\gamma\in \ak J_{facial}(X)$  and every point $x \in X\setminus \gamma$, we have
  $$
    x \in \Int(\gamma) \iff  h(x)\in \Int(h(\gamma));
  $$

It is obvious that $(E2) \Rightarrow (E2_F)$ and $(E3) \Rightarrow (E3_F)$.

\begin{lemma}\label{lem_condition_equivalence}
The following statements hold:
\begin{enumerate}
\item[(i)] $(E1) + (E2_F) \Longrightarrow (E2)$
\item[(ii)] $(E1) + (E2_F) + (E3_F) \Longrightarrow (E3)$
\item[(iii)] If $X$ is connected then 
$$ (E1) + (E2_F) \Longrightarrow (E3).$$
\end{enumerate}
\end{lemma}

\begin{proof}
{(i)}  Let $\gamma\in \ak J(X)$. We prove that
$$
w(h,\gamma)=c.
$$
We argue by induction on the number
$$
N(\gamma):=\#\{F:\ F \text{ is a face of }X\text{ and }F\subset
\Int(\gamma)\}.
$$

If  $N(\gamma) = 1$ (or equivalently,  $\gamma$ is a facial Jordan curve) then the conclusion follows immediately
from Condition $(E2_F)$.

Now suppose that $N(\gamma) \geq 2$. Choose an edge $e\subset \gamma$. Let $F$ be the face contained in $\Int(\gamma)$ whose boundary contains
$e$.
Let $\gamma_F$ be the facial Jordan curve associated with $F$. Thus
$$
F\subset \Int(\gamma_F)
        \quad \text{ and } \quad
\gamma_F\subset \partial F.
$$
Since $F\subset \Int(\gamma)$, we have
\textcolor{orange}{$$
\gamma_F \leq \gamma.
$$}
If $\gamma_F = \gamma$, then the result  again follows by Condition $(E2_F)$. So we may assume that $\gamma_F < \gamma$.
Because a neighborhood of a point in $e$ meets both $F$ and $\Ext(\gamma) \subset \ak F_\infty (\partial F)$, by Lemma \ref{graph-properties} (viii), 
$$e \subset \partial  F \cap \partial (\ak F_\infty (\partial F)) = \gamma_F.$$
Hence $\gamma$ and $\gamma_F$ share an arc $e$. 
Note that $e$ is not a loop edge since, otherwise, $\gamma_F=\gamma,$ a contradiction.
Let $\beta$ be the maximal common arc of $\gamma$ and $\gamma_F$ containing
$e$. 
Again, $\beta$ is not a loop edge so let $p,q$ be the endpoints of $\beta$.  
Write
$$
\gamma=\beta\cup\eta,
        \qquad
\gamma_F=\beta\cup\alpha,
$$
where $\eta=\gamma \setminus \mathring{\beta}$ and $\alpha=\gamma_F\setminus \mathring{\beta}$ are arcs with endpoints $p$ and $q$. 
Then
$$
\sigma:=\alpha\cup\eta
$$
is also a Jordan curve contained in $X$. Thus, the three arcs $\alpha,\beta,\eta$ form three Jordan curves containing $p$ and $q$. 
The largest curve is $\gamma$ and the two smaller curves are $\gamma_F$ and $\sigma.$

By construction, $\sigma < \gamma$, and
$\Int(\sigma)$ contains strictly fewer faces of $X$ than
$\Int(\gamma)$. Indeed, $F\not\subset \Int(\sigma)$.
Hence,
$$
N(\sigma)<N(\gamma).
$$
By the induction hypothesis,
$$
w(h,\sigma)=c.
$$
On the other hand, $\gamma_F$ is facial, so by $(E2_F)$,
$$
w(h,\gamma_F)=c.
$$

At the endpoints $p$ and $q$, the graph $X$ has at least three local
branches, namely the branches coming from $\alpha,\beta,\eta$. Hence $p$ and
$q$ are topological singular points of $X$. By $(E1)$,
$$
\ord(h,p)=\ord(h,q)=c.
$$
By Lemma \ref{ord-remove}, 
$$
\ord(h_{\alpha\cup \beta\cup \eta},p)=\ord(h|_{\alpha\cup \beta\cup \eta},q)=c.
$$
Applying Lemma~\ref{lem_theta} to the three arcs $\alpha,\beta,\eta$ gives $w(h, \gamma) = c$. This proves $(E2).$

{(ii)} By Item~{(i)}, both Conditions~$(E1)$ and~$(E2)$ hold. It remains to show that Conditions~$(E3)$ also hold.

Fix $\gamma\in \mathfrak J(X)$ and let $x\in X\setminus \gamma$ be such that $
x\in \Int(\gamma)$. We will prove that
$$
h(x)\in \Int(h(\gamma)).
$$
The converse implication follows by applying the same argument to $h^{-1}$.

Before proceeding, let us introduce some notation. Let $A,B\subset \mathbb R^2$ be compact definable sets of dimension $1$ such that $A\cap B=\varnothing$. Then $A$ is contained in a face of $B$. We write
$$
A<B
$$
if $A$ is contained in an inner face of $B$.

We distinguish two cases.

\medskip

\noindent
\textbf{{Case 1.}} \emph{$x$ and $\gamma$ are contained in the same connected component of $X$.}

By assumption, there exists an arc $\alpha \subset X$ joining $x$ to a point $p\in\gamma$ such that
$
\alpha\cap \gamma=\{p\}.
$
We may assume that 
$$\mathring{\alpha}\subset \Int(\gamma).$$
By Conditions $(E1)$ and $(E2)$,
$$
\ord(h,p)=c=w(h,\gamma).
$$
Hence, by Lemma~\ref{ord-add1},
$$
h(\mathring{\alpha})\subset \Int(h(\gamma)).
$$
In particular,
$$
h(x)\in \Int(h(\gamma)).
$$

\medskip

\noindent
\textbf{{Case 2.}} \emph{$x$ and $\gamma$ are not contained in the same connected component of $X$.}

Set
$$
Y:=X\cap \overline{\Int(\gamma)}.
$$
Let $Y_1$ be the connected component of $Y$ containing $x$, and choose a maximal chain
$$
Y_1<\cdots<Y_k
$$
of connected components of $Y$, where $Y_k$ is the connected component containing $\gamma$.
Since $x$ and $\gamma$ do not lie in the same connected component of $X$, we have $k\ge 2$.

Since $Y_{k-1}<Y_k$, there exists an inner face $F$ of $Y_k$ such that
$$
Y_{k-1}\subset F.
$$
Note that $F$ is not a face of $X$. 
Let $\gamma_F$ be the Jordan curve associated with $F$ by
Lemma~\ref{graph-properties} (viii). Then
$$
F\subset \Int(\gamma_F)
\qquad\text{and}\qquad
\gamma_F.
$$
In particular,
$$
x\in Y_1\subset F\subset \Int(\gamma_F).
$$

We claim that $\gamma_F$ is a facial Jordan curve of $X$.
Indeed, as $F$ is a face of $Y_k$, we have
$$
F\cap Y_k=\varnothing.
$$
Since $X\cap F$ is compact and disjoint from
$\partial F$,
$$
d:=\dist(\partial F,X\cap F)>0.
$$
Set 
$$U: = \{ y\in F: \dist(y, \partial F) \leq d/2\}.$$
Then $U\cap X=\varnothing$. Hence $U$ is contained in a face $F'$ of $X$. Since
$U\subset  F \subset \Int(\gamma_F)$,  $F' \subset \Int(\gamma_F)$ which is bounded, it is an inner face
of $X$. Moreover, by definition,  every point of $\partial F$ is a limit point of $U$, hence a limit point of $F'$. Thus, $\partial F \subset \partial F'$. Since $\gamma_F \subset \partial F$, it follows that 
$
\gamma_F\subset \partial F'.$
This, together with the fact $F' \subset \Int (\gamma_F)$ and Lemma~\ref{graph-properties} (viii)  (the uniqueness) yields $\gamma_F = \gamma_{F'}$. Hence
$\gamma_F\in \mathfrak J_{{facial}}(X)$, proving the claim.

Now, since
$$
\gamma_F\subset Y_k,
$$
the curve $\gamma_F$ lies in the same connected component of $X$ as $\gamma$.
Applying Case~1 to the Jordan curve $\gamma$ and any point of $\gamma_F$, we obtain
$$
h(\gamma_F)\subset \Int(h(\gamma)).
$$
Consequently, $$\Int(h(\gamma_F))\subset \Int(h(\gamma)).$$
Since $x\in \Int(\gamma_F)$ and $\gamma_F$ is a facial Jordan curve of $X$,
Condition ($E3_F$) implies that
$$
h(x)\in \Int(h(\gamma_F)).
$$
Therefore
$$
h(x)\in \Int(h(\gamma_F))\subset \Int(h(\gamma)).
$$
In both cases, we have shown that
$$
x\in \Int(\gamma)\Longrightarrow h(x)\in \Int(h(\gamma)).
$$
Applying the same argument to $h^{-1}$, we obtain the converse implication. Hence
Condition ($E3$) holds.

{(iii)} By Item~(i), we  can assume that the condition (E2) holds. Let $\gamma\in\ak  J(X)$. If $X\setminus \gamma = \varnothing$ then the conclusion is trivial. We may assume that $X\setminus \gamma \neq \varnothing$ and let $x\in X\setminus\gamma$.
Since $X$ is connected, there exist a point $p\in\gamma$ and a definable arc
$\alpha\subset X$ joining $x$ to $p$ such that $\alpha\cap\gamma=\{p\}$.
Then $\alpha\setminus\{p\}$ is contained in exactly one of the two components
$\Int(\gamma)$ or $\Ext(\gamma)$; in particular,
$$
x\in\Int(\gamma) \Longleftrightarrow \alpha\setminus\{p\}\subset\Int(\gamma).
$$
Choose a definable subset $Y\subset X$ such that the germ of $Y$ at $p$
has the following three branches: one is the germ of $\alpha$ at $p$, and
the other two are the germs at $p$ of the connected components of
$\gamma\setminus \{p\}$.
Then $\deg(p,Y)=3$. By Condition~($E1$), we have
$\ord(h,p)=c$, and therefore $\ord(h|_Y,p)=c$ by Lemma~\ref{ord-remove}.

From Condition ($E2$) we have $w(h,\gamma)=c$.
Therefore, by Lemma \ref{ord-add1}, the branch $\alpha\setminus\{p\}$ lies in $\Int(\gamma)$ if and only if
$h(\alpha)\setminus\{h(p)\}$ lies in $\Int(h(\gamma))$.
Hence
$$
x\in\Int(\gamma) \Longleftrightarrow h(x)\in\Int(h(\gamma)),
$$
which is Condition $(E3)$.
\end{proof}

\begin{lemma}\label{same-boundary}
Let $X$ be a compact definable set of dimension at most $1$, and let
$h:X\to \mathbb R^2$ be a definable embedding satisfying conditions
$(E1)$--$(E3)$. Then, for every face $F$ of $X$, there exists a
unique face $F'$ of $h(X)$ such that
$$
\partial F' = h(\partial F).
$$
In particular:
\begin{enumerate}
    \item[(i)] $F$ is the outer face of $X$ if and only if $F'$ is the
    outer face of $h(X)$;

    \item[(ii)] $F$ is an inner face of $X$ if and only if $F'$ is an
    inner face of $h(X)$.
\end{enumerate}
\end{lemma}

\begin{proof}
{(i)}
Let $F$ be the outer face of $X$ and let $F'$ be the outer face of $h(X)$.
We prove $\partial F' = h(\partial F)$.

Let $p\in\partial F$.
By Lemma \ref{graph-properties} (x), if $h(p)\notin\partial F'$, there exists a Jordan curve
$\sigma\in\ak  J(h(X))$ such that $h(p)\in\Int(\sigma)$.
By Condition ($E3$), $p\in\Int(h^{-1}(\sigma))$.
Hence, By Lemma \ref{graph-properties} (x), $p\not\in\partial F$, a contradiction.
Therefore, $h(p)\in\partial F'$, and $h(\partial F)\subset \partial F'$.

Conversely, by applying the same argument to $h^{-1}$, we obtain $\partial F'\subset h(\partial F)$. Hence, $\partial F'=h(\partial F)$. The uniqueness of $F'$ follows from the uniqueness of the outer face.

\smallskip
{(ii)}
Now let $F$ be an inner face of $X$.
By Lemma~\ref{graph-properties} (viii),  there exists a unique
$\gamma\in \ak  J(\partial F)$ such that $F\subset \Int(\gamma)$.
Choose an edge $e$ of $X$ contained in $\gamma$.
Then $h(e)$ is an edge of $h(X)$ contained in the Jordan curve $h(\gamma)$.

Let $F'$ be the face of $h(X)$ such that $h(e)\subset\partial F'$ and
$F'\subset\Int(h(\gamma))$.
We will show that $\partial F'=h(\partial F)$.
Let $\gamma'\in\ak  J(\partial F')$ be the (unique) Jordan curve such that
$F'\subset\Int(\gamma')$ in view of Lemma \ref{graph-properties} (viii).

\begin{claim}\label{same-border}
We have $\gamma'=h(\gamma)$.
In particular, $h(\gamma)\subset\partial F'$ and $F'\subset\Int(h(\gamma))$.
\end{claim}

\begin{proof}
Note that $\partial F' \subset \overline{\Int(h(\gamma))}$, and hence
$\gamma' \subset \overline{\Int(h(\gamma))}$.
By Lemma \ref{BS}, this implies that $\gamma' \leqslant h(\gamma)$.

We now show that $\gamma' = h(\gamma)$.
Suppose, for contradiction, that $\gamma' < h(\gamma)$.
Then, by Condition~($E3$) we obtain
$h^{-1}(\gamma') < \gamma$.
Since $\gamma \subset \partial F$, it follows that $F \subset \Ext(
h^{-1}(\gamma'))$.
On the other hand, we have
$h(e) \subset \partial F' \subset \overline{\Int(\gamma')}$,
and therefore either $h(e) \subset \gamma'$ or
$h(\mathring e) \subset \Int(\gamma')$.

If the former holds, then $e \subset h^{-1}(\gamma')$; together with
$h^{-1}(\gamma') < \gamma$ and Lemma  \ref{graph-properties} (vii), this implies
$F \subset \Int(h^{-1}(\gamma'))$, contradicting the fact that $F\subset \Ext(h^{-1}(\gamma'))$.

If the latter holds, then $\mathring e \subset \Int(h^{-1}(\gamma'))$.
Since $e \subset \partial F$, this again yields
$F \subset \Int(h^{-1}(\gamma'))$, leading to the same contradiction.

Therefore, $\gamma' = h(\gamma)$ and the claim follows.
\end{proof}

\smallskip
 \textbf{Inclusion $h(\partial F)\subset \partial F'$.}
Suppose for contradiction that there exists $p\in\partial F$ with
$h(p)\notin\partial F'$.
Then $h(p)\notin h(\gamma)$, and since $p\in\partial F\subset \overline{\Int(\gamma)}$,
we get $h(p)\in\Int(h(\gamma))$. 

Let $\{e_i\}_{i\in I}$ be the edges of $h(X)$ containing $h(p)$, and let
$$
Y:=h(X)\setminus\Bigl(\{h(p)\}\cup\bigcup_{i\in I}\mathring e_i\Bigr).
$$
Then $h(p)$ lies in an inner face $F''$ of $Y$. Since $h(p)\notin \partial F'$, it follows that $\mathring{e}_i\not\subset \partial F'$. Therefore, removing $\mathring{e}_i$ from $X$ does not affect the boundary $\partial F'$. Consequently, $F'$ remains an inner face of $Y$. Hence, $F'$ and $F''$ are different inner faces of $Y$.
By Lemma \ref{graph-properties} (ix), there exists
$\sigma\in\ak  J(Y)\subset\ak  J(h(X))$ separating $F'$ and $F''$.
Thus, either:
\begin{enumerate}
\item[(*)] $F'\subset \Int(\sigma)$ and $h(p) \in \Ext(\sigma)$, or
\item[(**)] $F' \subset \Ext(\sigma)$ and $h(p)\in \Int(\sigma)$.
\end{enumerate}

Assume that (*) holds. It follows that $\partial F' \subset \overline{\Int(\sigma)}$, which yields $h(\gamma)\le \sigma$.
By Condition~(E3) we get
$\gamma\le h^{-1}(\sigma)$.
Therefore,
$$
p\in\partial F\subset \overline{\Int(\gamma)}\subset \overline{\Int(h^{-1}(\sigma))}.
$$
Since $h(p) \in \Ext(\sigma)$, by Condition ($E3$), $p \in \Ext (h^{-1}(\sigma))$, which is a contradiction. 

If (**) holds, then by Condition~($E3$), we have
$
p\in \Int(h^{-1}(\sigma)).
$
Since $p\in \partial F$, it follows that
$
F\subset \Int(h^{-1}(\sigma)),
$
and hence
$
\gamma\le h^{-1}(\sigma).
$
Applying Condition~($E3$) again, we obtain
$
h(\gamma)\le \sigma.
$
By Claim~\ref{same-border}, we have
$
F'\subset \Int(h(\gamma))\subset \Int(\sigma),
$
which contradicts the assumption that
$
F'\subset \Ext(\sigma).
$

Consequently
$
h(\partial F)\subset \partial F'.
$

\smallskip
\textbf{Inclusion $\partial F'\subset h(\partial F)$.}
Applying the previous argument to $h^{-1}$ (which satisfies the same assumptions), we get $h^{-1}(\partial F')\subset\partial F$, i.e., $\partial F'\subset h(\partial F)$. 

Combining both inclusions gives $\partial F'=h(\partial F)$.

\smallskip

\textbf{Uniqueness of $F'$.} Suppose that there is another inner face $F^*$ of $h(X)$ such that 
$\partial F^*=\partial F'$.  By Claim \ref{same-border}, we have $h(\gamma)=\gamma' \subset \partial F'$. Choose an edge $e\subset\gamma'$. Since $e$ lies on a Jordan curve, Lemma \ref{graph-properties} (ii) implies that $e$ is contained in the boundary of exactly two faces of
$X$, one on each side of $\gamma'$. Both $F'$ and $F^*$ are contained in
$\Int(\gamma')$, so both must be the unique face adjacent to $e$ on the inner
side of $\gamma_F$. Hence $F^*=F'$.
Therefore, $F'$ is unique.
\end{proof}

\begin{proposition}\label{Join-Bound}
Let $X\subset \bb R^2$ be a compact definable set of dimension $\leq 1$ and $h:  X \to \bb R^2$ be a definable embedding.  Assume that Conditions~$(E1)$--$(E3)$ hold. Then, there exist a compact, connected definable set $\widetilde{X} \subset \bb R^2$ with $\dim \widetilde{X} = 1$,
$X \subset \widetilde X$ and a definable embedding
$\widetilde h\colon\widetilde X\to\mathbb R^2$ extending $h$
such that the following conditions are satisfied:
\begin{enumerate}
\item [(i)] There exists $R>0$ such that:
\begin{itemize}
\item $\mathbb S_R\subset \widetilde X\cap \widetilde h(\widetilde X)$;
\item $\widetilde X\cup \widetilde h(\widetilde X)\subset \overline{\mathbb B}_R$.
\item $X\cup h(X)\subset\Int(\mathbb S_R)$;
\item For $(x_1, x_2) \in \bb S_R$
$$
\widetilde{h}(x_1,x_2)=
\left\{
\begin{array}{rl}
(x_1,x_2), & \text{if } c=1\\
(-x_1,x_2), & \text{if } c=-1.
\end{array}
\right.
$$
\end{itemize}

\item [(ii)] Conditions $(E1)-(E3)$ hold for $\widetilde h$;
\end{enumerate}
\end{proposition}

\begin{proof}
Choose $R>0$ such that both $X$ and $h(X)$ are contained in $\Int(\mathbb{S}_R)$.
Set $X_0 := X \cup \mathbb{S}_R$. We define an extension
$h_0 : X_0 \to \mathbb{R}^2$ of $h$ by setting $h_0|_X = h$ and, for $(x_1,x_2)\in\mathbb{S}_R$,
$$
h_0(x_1,x_2)=
\left\{
\begin{array}{rl}
(x_1,x_2), & \text{if } c=1\\
(-x_1,x_2), & \text{if } c=-1.
\end{array}
\right.
$$
It is clear that $h_0$ is a definable embedding satisfying
Conditions~(E1)--(E3).

We now iteratively connect the connected components of $X_0$
(and of $h_0(X_0)$) by adding definable arcs  internally contained in $\Int(\bb S_R)$  and extending $h_0$
along these arcs. At each step, the number of connected components
is reduced by one. The process stops once the resulting sets are connected.

Let $n$ be the number of connected components of $X_0$. We will prove by induction the following statement: for each $i = 0, \ldots, n-1$ there is a definable set $X_i$ such that : 
\begin{itemize}
    \item $X_{i-1} \subset X_i$
    \item $X_i \subset \overline{\Int(\bb S_R)}$
    \item $\ak J(X_{i-1}) = \ak J(X_i)$
    \item $X_i$ has $(n-i)$ connected components
    \item there is a  definable embedding  $h_i: X_i \to \bb R^2$ extending $h_{i-1}$ such that Conditions $(E1)$--$(E3)$ are satisfied.
\end{itemize}
Then $\widetilde{X}:=X_{n-1}$ is connected, and $\widetilde{h}:=h_{n-1}$ is the desired extension of $h$.
Here, by convention, we set $X_{-1}=X_0$ and $h_{-1}=h_0$.

For $i=0$, the construction is trivial. Assume that $X_i$ and $h_i:X_i\to\bb R^2$, with $0\le i<n-1$, have been constructed. We now show how to construct $X_{{}i+1}$ and $h_{i+1}$.

Let $F_i$ be a face of $X_i$ whose boundary intersects at least two
different connected components of $X_i$, say $X_{i,1}$ and $X_{i,2}$.
Set
$$
Z_{i,j} := X_{i,j} \cap \partial F_i \qquad (j=1,2).
$$
For each $j\in\{1,2\}$, choose a point $p_{i,j}\in Z_{i,j}$ as follows:
  \begin{itemize}
  \item if $Z_{i,j}$ contains a leaf, choose $p_{i,j}$ to be such a leaf. Recall that a leaf means either a vertex of degree $0$ (isolated
point) or a vertex of degree $1$.
  \item otherwise, choose $p_{i,j}$ to be a point lying on a Jordan curve
        contained in $Z_{i,j}$ at which $X_i$ is topologically regular.
  \end{itemize}

Let $\alpha_{i}$ be a definable arc internally contained in $F_i$ joining
$p_{i,1}$ and $p_{i,2}$.
Let $F_i'$ be the unique face of $h_i(X_i)$ corresponding to $F_i$ obtained by Lemma \ref{same-boundary}. Then
$$
h_i(X_{i,j}) \cap \partial F_i'
= h_i(X_{i,j}) \cap h_i(\partial F_i)
= h_i(X_{i,j}\cap \partial F_i)
= h_0(Z_{i, j}).
$$
Hence there exists a definable arc $\beta_i$ internally contained in  $F_i'$
connecting $h_i(p_{i,1})$ and $h_i(p_{i,2})$.

Define $X_{i+1} := X_i \cup \alpha_i$ and define
$h_{i+1} : X_{i+1} \to \mathbb{R}^2$ by setting $h_{i+1}|_{X_i}=h_i$
and letting $h_{i+1}|_{\alpha_i}$ be a definable homeomorphism
from $\alpha_i$ onto $\beta_i$.
By construction, $X_{i+1}$ has exactly one fewer connected component
than $X_i$. 
We now verify that Conditions $(E1)$--$(E3)$ 
still hold for $X_{i+1}$ and $h_{i+1}$.

Note that Condition~$(E2)$ holds automatically since $\ak J(X_{i+1}) = \ak J(X_i)$ in view of Lemma \ref{graph-properties} (v).

For Condition $(E3)$, since no new Jordan curve is created, it suffices to check points lying on $\mathring{\alpha}_i$.
Let $p\in \mathring{\alpha}_i$, and let $\gamma\in\mathfrak{J}(X_{i+1})$
be such that $p\in \Int(\gamma)$. We need to show that
$$
h_{i+1}(p)\in \Int(h_{i+1}(\gamma)).
$$
Since $F_i$ is a face and $\mathring{\alpha}_i\subset F_i$,
we have $F_i\subset \Int(\gamma)$.
Let $\gamma_{F_i}$ be the facial Jordan curve associated to $F_i$, in light of Lemma \ref{graph-properties} (viii). 
We have  $\gamma_{F_i} \leq  \gamma$, and therefore $h_{i+1}(\gamma_{F_i}) \leq h_{i+1}(\gamma)$ (since $\gamma_{F_i}$ and $\gamma$ are both contained in $X_0$, and $h_{i+1}|_{X_0}=h_{0}|_{X_0}$ satisfies Condition ($E3$)). 
Now, by Claim \ref{same-border}, $h(\gamma_{F_i}) = \gamma_{F_{i}'}$. Therefore, $\mathring{\beta}_i
\subset F_i' \subset \Int(\gamma_{F_i'})
\subset \Int(h_{i+1}(\gamma))$, and hence $h(p) \in \mathring{\beta}_i \subset  \Int(h_{i+1}(\gamma))$. 
This proves Condition $(E3)$.

It remains to show that Condition $(E1)$ holds.  
The only points at which new topological singularities may arise
after adding $\alpha_i$ are the endpoints $p_{i,1}$ and $p_{i,2}$.
Fix $j\in\{1,2\}$ and suppose that $p_{i,j}$ is a topological
singular point of $X_{i+1}$ with degree at least $3$.
By construction, $p_{i,j}$ is a topologically regular point of $X_i$
lying on a Jordan curve $\gamma\subset Z_{i,j}$, and hence
$
\deg(p_{i,j},X_{i+1})=3.
$
By Condition~$(E3)$,
$$
\mathring{\alpha}_i\subset \Int(\gamma)
\Longleftrightarrow
\mathring{\beta}\subset \Int(h(\gamma)).
$$
Combining this with Lemma~\ref{ord-add1}, we obtain
$$
\ord(h_{i+1},p_{i,j})
= w(h_{i+1},\gamma)
= w(h_i,\gamma)
= c.
$$
\end{proof}

\begin{theorem}\label{thm_main}
    Let $X \subset \bb R^2$ be a compact definable set of dimension $\leq 1$ and let
$h: X \to \mathbb{R}^2$ be a definable embedding. Then, $h$ has a definable homeomorphic extension on the whole of $\bb R^2$ if and only if there is a constant $c\in \{-1, 1\}$ such that Conditions $(E1)$, $(E2_F)$ and $(E3_F)$ hold. 
\end{theorem}

\begin{proof}
The necessity is clear. We now prove the sufficiency. By
Lemma~\ref{lem_condition_equivalence}, it is enough to assume that Conditions $(E1)$--$(E3)$ hold. Moreover, by Proposition~\ref{Join-Bound}, we may
assume that $X$ and $h(X)$ are connected, and that there exists $R>0$ such that
\begin{itemize}
\item $\mathbb S_R\subset X\cap  h(X)$;
\item $X$ and $h(X)$ both are contained in $\overline{\Int(\mathbb S_R)}$;
\item For $(x_1, x_2) \in \bb S_R$
$$
{h}(x_1,x_2)=
\left\{
\begin{array}{rl}
(x_1,x_2), & \text{if } c=1\\
(-x_1,x_2), & \text{if } c=-1;
\end{array}
\right.
$$
\end{itemize}

\begin{claim}\label{claim_jordan_region}
There exist a compact connected definable set
$\widetilde X\subset \overline{\Int(\mathbb S_R)}$, with
$X\subset \widetilde X$, and a definable embedding
$$
\widetilde h:\widetilde X\to \mathbb R^2
$$
extending $h$, with the following properties:
\begin{enumerate}
    \item [(i)]every inner face of $\widetilde X$ and of
    $\widetilde h(\widetilde X)$ is a Jordan region;
    \item [(ii)]$\widetilde h$ maps the boundary of each face of
    $\widetilde X$ onto the boundary of a face of
    $\widetilde h(\widetilde X)$.
\end{enumerate}
\end{claim}

Let us first explain why Claim~\ref{claim_jordan_region} implies that $h$
extends to the whole plane. By Claim~\ref{claim_jordan_region}, the map
$\widetilde h$ sends, inside $\Int(\mathbb S_R)$, the boundary of each face
of $\widetilde X$ homeomorphically onto the boundary of a face of
$\widetilde h(\widetilde X)$. Since these faces are Jordan regions, the
definable Schoenflies theorem (Theorem~\ref{Thm_definable_Schoenflies})
provides an extension of $\widetilde h$ from each inner face of
$\widetilde X$ to the corresponding inner face of
$\widetilde h(\widetilde X)$. We denote the resulting extension by $h'$.

Outside $\Int(\mathbb S_R)$, we define
$$
h'(x_1,x_2)=
\left\{
\begin{array}{rl}
(x_1,x_2), & \text{if } c=1\\
(-x_1,x_2), & \text{if } c=-1.
\end{array}
\right.
$$
Then $h':\mathbb R^2\to \mathbb R^2$ is clearly a definable homeomorphism
extending $h$.

It remains to prove Claim \ref{claim_jordan_region}. 

\begin{proof}[Proof of Claim~\ref{claim_jordan_region}]
    If $X$ has no leaves and has no special Jordan curves, we take $\widetilde{X} = X$ and $\widetilde{h} = h$. Item~(i) then follows from  Theorem \ref{uojr}, and Item~(ii) follows from  Lemma \ref{same-boundary}.
Hence, we may assume that $X$ has at least one leaf in $\Int(\mathbb S_R)$ or at least one special Jordan curve.

We will modify $X$ (and simultaneously $h(X)$) by adding pairwise disjoint definable arcs to eliminate all leaves and all special Jordan curves, thus partitioning
$\Int(X)$ (and $\Int(h(X))$) into Jordan regions. We argue by induction on
$$
k:=\#\{\text{leaves of }X \} \;+\; \#\{\text{special Jordan curves of }X\}.
$$
The case $k=0$ is clear from the previous paragraph. Assume that the statement holds if $k<n$, and consider the case $k=n$.

If $X$ has a leaf in $\Int(\mathbb S_R)$, choose such a leaf and denote it by $p$.
Let $F$ (resp.\ $F'$) be the inner face of $X$ (resp.\ $h(X)$) whose boundary contains
$p$ (resp.\ $h(p)$).
Otherwise, $X$ has a special Jordan curve, then choose a special Jordan curve $S$ of $X$ and pick a topologically
regular point $p\in S$. Then $h(p)$ is a topologically regular point of $h(X)$ lying on the
special Jordan curve $h(S)$. Let $F$ (resp.\ $F'$) be the face of $X$ (resp.\ $h(X)$) which is not contained
in $\Int(S)$ (resp.\ $\Int(h(S))$) and whose boundary contains $p$ (resp.\ $h(p)$).
By the definition of special Jordan curves, both $F$ and $F'$ are bounded.

Let $\gamma_F$ be the facial Jordan curve associated to $F$. Choose a topologically regular point $\widetilde p\in X\cap \gamma_F$. Note that $\widetilde{p}\ne p.$ 
Pick a definable arc $\alpha$ joining $p$ to $\widetilde p$ with $\mathring\alpha\subset F$.
Similarly, pick a definable arc $\sigma$ joining $h(p)$ to $h(\widetilde p)$ with $\mathring\sigma\subset F'$.
Define
$$
h_1\colon X\cup \alpha \longrightarrow h(X)\cup \sigma
$$
by setting $h_1|_X=h$ and choosing $h_1|_\alpha\colon \alpha\to \sigma$ to be a definable homeomorphism.

To apply the induction hypothesis, it suffices to verify that
Conditions $(E1)$ and $(E2)$ still hold for the pair $(X\cup\alpha,\,h_1)$.

Let us show that the number of leaves and special Jordan curves of $X\cup\alpha$ is at most $n-1$.
Indeed, if $p$ is a leaf of $X$, then $p$ becomes a topologically regular point of $X\cup\alpha$,
and hence the number of leaves decreases by one.
If $p$ lies on a special Jordan curve $S$, then adding $\alpha$ creates at least two points on $S$ at which
an incident edge is not contained in $\overline{\Int(S)}$.
Consequently, $S$ is no longer a special Jordan curve in $X\cup\alpha$.
It remains to show that no new special Jordan curve is created.
Indeed, let $\beta\subset X\cup\alpha$ be a Jordan curve containing $\alpha$.
Then $\beta$ must meet $X$ at $p$ and $\widetilde p$.
Moreover, at each of these points, the germ of $X\setminus\{\cdot\}$ has a component not contained in $\Int(\beta)$.
Therefore, by definition, $\beta$ is not a special Jordan curve.


 To apply the induction hypothesis and obtain Item~(ii), it suffices to show that Conditions $(E1)$ and $(E2_F)$ hold for $\widetilde{X}$ and $\widetilde{h}$, then apply Lemma~\ref{same-boundary}. 

Adding $\alpha$ can create topological singularities of degree $\ge 3$ only at $\widetilde p$
and possibly at $p$ (in the case where $p$ lies on a special Jordan curve).
Since 
$$\alpha\setminus\{\widetilde p\}\subset \Int(\gamma_F)\quad \text{ and } \quad \sigma\setminus\{h_1(\widetilde p)\}\subset \Int(h(\gamma_F))=\Int(h_1(\gamma_F)),$$
Lemma \ref{ord-add1} gives
$$
\ord(h_1,\widetilde p)=w(h_1,\gamma_F)=w(h,\gamma_F)=c.
$$
If $p$ becomes singular, then $p$ lies on a special Jordan curve $S$ and the arcs $\alpha\setminus\{p\}$ and $\sigma\setminus
\{h_1(p)\}$ lie outside $\Int(S)$ and $\Int(h(S))$, respectively. 
Hence, Lemma \ref{ord-add1} again yields $\ord(h_1,p)=c$. 
Thus, Condition $(E1)$ holds.

 It remains to show that $(E2_F)$ holds. 
By construction, $\alpha$ (resp., $\sigma = h(\alpha)$) separates $F$ (resp., $F'$) into two new inner faces, each of which shares a common arc with $\gamma_F$. Let $\beta$ be the facial Jordan curve associated with one of these inner faces. Then $\beta < \gamma_F$. Let $\eta$ be a common arc of $\beta$ and $\gamma_F$. It follows that
$$
\eta \subset \overline{\Int(\gamma_F)\cup \Int(\beta)}.
$$
Fix an orientation on $\beta$, and choose an orientation on $\gamma_F$ so that both curves induce the same orientation on $\eta$. By Lemma~\ref{InOut}, we have
\begin{equation}\label{CL3.6.1}
    o(\gamma_F,\beta)=1.
\end{equation}
Since $h_1$ is an embedding, it preserves the orientation of arcs. Therefore, the induced orientations on $h_1(\gamma_F)$ and $h_1(\beta)$ agree on the arc $h_1(\eta)$. Since $h_1(\beta) < h_1(\gamma_F)$, we obtain
$$
h_1(\eta)\subset \overline{\Int(h_1(\beta))\cup \Int(h_1(\gamma_F))}.
$$
Applying Lemma~\ref{InOut} again, we get
\begin{equation}\label{CL3.6.2}
o\bigl(h_1(\gamma_F),h_1(\beta)\bigr)=1.    
\end{equation}
Note that,
$$
w(h_1,\gamma_F)=w(h,\gamma_F)=c.
$$
It then follows that
$$
w(h_1,\beta)=c.
$$
This, together with~\eqref{CL3.6.1} and~\eqref{CL3.6.2}, implies that $(E2_F)$ holds.
\end{proof}
This ends the proof of the theorem.
\end{proof}

The following results are immediate from Lemma \ref{lem_condition_equivalence} and  Theorem \ref{thm_main}.

\begin{corollary}\label{cor_main2}
Let $X \subset \bb R^2$ be a compact connected definable set of dimension $\leq 1$ and let
$h: X \to \mathbb{R}^2$ be a  definable embedding. Then, $h$ has a definable homeomorphic extension on the whole of $\bb R^2$ if and only if there is a constant $c\in \{-1, 1\}$ such that Conditions $(E1)$ and $(E2_F)$  hold. 
\end{corollary}

\section{Topological extension: the unbounded case}\label{section4}

Let $a\in \bb R^2$, we call the following map 
$$ \phi_a: \bb R^2 \setminus \{a\} \to \bb R^2,\quad x\mapsto \frac{ x-a}{\|x - a\|^2}$$ the \textit{compactification} with respect to $a$. 

Let $X\subset \mathbb R^2$ be an unbounded definable set of dimension $1$. As before, we denote by $\ak J(X)$ the set of Jordan curves contained in $X$. A bounded face, or inner face, of $X$ is a bounded connected component of $\mathbb R^2\setminus X$. Similarly to Lemma \ref{graph-properties} (viii), for each inner face $F$ of $X$, there exists a unique Jordan curve, denoted by $\gamma_F$, such that
$$
\gamma_F\subset \partial F
\qquad\text{and}\qquad
F\subset \operatorname{Int}(\gamma_F).
$$
We still denote by $\ak J_{facial}(X)$ the set of all such Jordan curves. Thus Conditions $(E1)$, $(E2)$, $(E3)$, $(E2_F)$ and $(E3_F)$ are well-defined for $X$.

By Hardt's triviality theorem \cite{Hardt1980, Dries1998}, there exists $R_0>0$ such that for every $R\ge R_0$
the homeomorphism type of $X\setminus \mathbb B_R$ is constant. In particular, 
$$
\deg(\infty,X):=\text{ the number of connected components of } X\setminus \mathbb B_R
$$
is independent of $R\ge R_0$. We call $\deg(\infty,X)$ the \emph{degree of $X$ at infinity}.

Fix $R\ge R_0$, set $m:=\deg(\infty,X)$ and denote by
$$
\{X_1^\infty,\dots,X_m^\infty\}
$$
the connected components of $X\setminus \mathbb B_R$. We refer to this collection as the \emph{germ of $X$ at infinity}.

A \emph{Jordan curve around $\infty$ (with respect to $X$)} is a definable Jordan curve
$$
\gamma\subset \mathbb R^2\setminus \mathbb B_R
$$
such that for each $i=1,\dots,m$,
the intersection $X_i^\infty\cap \gamma$ consists of exactly one point. Let $p_i = \gamma \cap X^\infty_i$. 

First, we suppose that  $m\geq 3$.  Then,  the counterclockwise orientation of $\gamma$ induces a cyclic order on the set
$\{p_1,\dots,p_m\}$, hence on $\{X_1^\infty,\dots,X_m^\infty\}$ as follows.
For distinct indices $i,j\in\{1,\dots,m\}$, we write 
$$X_i^\infty < X_j^\infty \quad \text{if and only if} \quad p_i <_\gamma p_j.$$ 
It can be shown (as in the local case) that this cyclic order on
$\{X_1^\infty,\dots,X_m^\infty\}$ does not depend on the choice of the Jordan curve $\gamma$ around $\infty$.
We call it the \emph{cyclic order of $X$ at infinity}.

Let $h\colon X\to \mathbb R^2$ be a definable, continuous, injective and coercive map. Then, $h$ induces a well-defined map on germs at infinity
$$
h_\infty:\{X_1^\infty,\dots,X_m^\infty\}\longrightarrow \{h(X_1^\infty),\dots,h(X_m^\infty)\}.
$$

 We say that $h$ \emph{preserves the cyclic order at infinity} if 
$$X_i^\infty < X_j^\infty \implies h(X_i^\infty) < h(X_j^\infty) \quad \text{ for all } 1 \leq i\ne j \leq m,$$
and $h$ \emph{reverses the cyclic order at infinity} if 
$$X_i^\infty < X_j^\infty \implies h(X_j^\infty) < h(X_i^\infty) \quad \text{ for all } 1 \leq i\ne j \leq m,$$
In the former case, we set $\ord(h,\infty)=1$, and in the latter case, we set $\ord(h,\infty)=-1$; otherwise, we set $\ord(h, \infty) := 0$. 

In what follows, we explain how to choose points
$p\in \mathbb R^2\setminus X$ and $q\in \mathbb R^2\setminus h(X)$ so that the extension problem for $h$ on the unbounded set $X$ can be reduced to the extension problem for the definable embedding
$$
g_{p,q}(x)=\left\{
\begin{array}{rl}
\phi_q\circ h\circ \phi_p^{-1}(x), & \text{if } x\neq 0\\
0, & \text{if } x=0
\end{array}\right.
$$
on the compact definable set $\phi_p(X)\cup\{0\}$, which can be solved by applying Theorem~\ref{thm_main}.

\noindent\textbf{Case 1:} $\deg (\infty, X) = 1$. 

In this case,
$\bb R^2 \setminus X$ has only one connected component at infinity, we denote this by $U_1$. We denote $V_1$ the connected component at infinity of $\bb R^2 \setminus h(X)$. Pick $p \in U_1$ and $q \in V_1$. 

\noindent\textbf{Case 2: }$\deg (\infty, X) = 2$. 

In this case, the germ of $\bb R^2 \setminus X$  has two connected components at infinity, we denote these germs by $U_1$ and $U_2$. Similarly, we will denote by $V_1$ and $V_2$ respectively the connected components of the germ of $\bb R^2\setminus h(X)$ at infinity. We will consider two possibilities of the pair $(p, q)$, first  $p \in U_1$ and $q\in V_1 $, second $p \in U_1$ and $q \in V_2$. 

\noindent\textbf{Case 3: $\deg(\infty, X)\geq 3$.}

By reindexing, we can arrange the sets $X_i^\infty$ in cyclic order as follows:
$$
X_1^\infty < \cdots < X_m^\infty.
$$

Let $U_i$ denote a connected component at infinity bounded by $X_i$ and $X_{i+1}$.

If $w(h, \infty) = c \in \{\pm 1\}$, then either 
$h(X_i^\infty)<h(X_{i+1}^\infty)$ or 
$h(X_{i+1}^\infty)<h(X_{i}^\infty)$, so a connected component at infinity bounded by
$h(X_i^\infty)$ and $h(X_{i+1}^\infty)$ is well-defined. In this case,  we denote such a connected component by $V_i$.  

Pick $i\in\{1,\dots,m\}$ and choose
$$p\in U_i \qquad\text{and}\qquad q\in V_i. $$

\begin{remark}\label{Rem41}\rm
By definition and Lemma~\ref{PreserveReverse}, if $h$ admits a homeomorphic
extension to the whole plane, say $H_0:\mathbb R^2\to\mathbb R^2$, then
$w(h,\infty)=c$ for some $c\in\{\pm1\}$. Moreover, if $m\neq 2$, then
$H_0(U_i)=V_i$ for every $i=1,\ldots,m$. If $m=2$, then either
$H_0(U_1)=V_1$ or $H_0(U_1)=V_2$.
\end{remark}

Now, for the points $p$ and $q$ chosen as above, we  consider the following definable embedding $g_{p, q}: \phi_p(X) \cup \{0\} \to \mathbb R^2$ defined by 
$$
g_{p,q}(x)=\left\{
\begin{array}{rl}
\phi_q\circ h\circ \phi_p^{-1}(x), & \text{if } x\neq 0\\
0, & \text{if } x=0
\end{array}\right.
$$

The main theorem in this section is as follows: 
\begin{theorem}[Reduction to the bounded case]\label{thm_main_unbounded}
Let $X\subset \mathbb R^2$ be a closed unbounded definable set of dimension
at most $1$ and let
$$
h:X\to \mathbb R^2
$$
be a coercive definable embedding. Then, $h$ extends to a definable
homeomorphism of $\mathbb R^2$ if and only if
$$
\ord(h,\infty)\in\{\pm1\}
$$
and $g_{p, q}$
extends to a definable homeomorphism of $\mathbb R^2$. In the case $\deg(\infty, X)=2$, this means that  one of the two choices of the pair $(p,q)$ described above gives such an extension.
\end{theorem}

\begin{lemma}\label{lem_move_point_extension}
Let $h: X \to \mathbb R^2 $ be as in Theorem~\ref{thm_main_unbounded}, and suppose that $h$
extends to a definable homeomorphism of $\mathbb R^2$. For the points $p$
and $q$ chosen as above, there exists such an extension
$$
H:\mathbb R^2\to\mathbb R^2
$$
with $H(p)=q$.
\end{lemma}

\begin{proof} 

Let $H_0: \bb R^2 \to \bb R^2$ be a definable homeomorphic extension of $h$. Then, $H_0$ maps the connected component containing $p$ to the connected component containing $q$. Without loss of generality, we may assume $p \in U_1$ and $q\in V_1$. 

Choose $p_0 \in U_1\setminus\{ p\}$ such that $q_0: = H_0(p_0) \neq q$. 
Set
$$
Y_{p_0}:=\phi_{p_0}(X)\cup\{0\}.
$$
Define
$$
g_0:Y_{p_0}\to\mathbb R^2
$$
by
$$
g_0(x)=\left\{
\begin{array}{rl}
\phi_{q_0}\circ h\circ \phi_{p_0}^{-1}(x), & \text{if } x\neq 0\\
0, & \text{if } x=0.
\end{array}\right.
$$
Since $H_{0}(p_0)=q_0$, the map
$$
G_0(x):=\left\{
\begin{array}{rl}
     \phi_{q_0}\circ H_0\circ\phi_{p_0}^{-1}(x), & \text{if } x\neq 0  \\
     0, & \text{if } x=0
\end{array}\right.
$$
is a definable homeomorphism of $\mathbb R^2$ extending $g_0$.

Now set
$$
\widetilde p:=\phi_{p_0}(p)
\quad \text{ and } \quad
\widetilde q:=\phi_{q_0}(q).
$$
Since $p,p_0$ are in the same connected component of $\mathbb R^2 \setminus X$ the point $\widetilde p$ lies on the outer face of $Y_{p_0}.$
Similarly, since $q,q_0$ lie in the same connected component of $\mathbb R^2 \setminus h(X)$, the point $\widetilde q$ lies on the outer face of $g_0(Y_{p_0}).$

Define
$$
\widetilde g_0:Y_{p_0}\cup\{\widetilde p\}\to\mathbb R^2
$$
by
$$
\widetilde g_0|_{Y_{p_0}}=g_0
\quad \text{ and } \quad
\widetilde g_0(\widetilde p)=\widetilde q.
$$
We claim that $\widetilde g_0$ satisfies the bounded extension criterion required in Theorem~\ref{thm_main}.

Note that, since $G_0$ is a definable homeomorphism extending $g_0$ on the whole of $\mathbb{R}^2,$ in light of Theorem ~\ref{thm_main}, it follows that $g_0$ satisfies Conditions $(E1),\ (E2_F)$ and $(E3_F)$ in Section \ref{section3}. 
Now, adding the isolated point $\widetilde p$ in the outer face of $Y_{p_0}$ creates no new Jordan curves and no new topological singular points of degree at least $3$. Therefore, Condition $(E1)$ and $(E2_F)$ are automatically satisfied for $\widetilde{g}_0$.

It remains to check Condition $(E3_F)$ for $\widetilde{g}_0$.
Let
$$
\gamma\in \ak  J_{{facial}}(Y_{p_0}\cup\{\widetilde p\}).
$$
Since $\widetilde p$ is isolated and lies in the outer face of $Y_{p_0}$, every facial Jordan curve of
$Y_{p_0}\cup\{\widetilde p\}$ is already a facial Jordan curve of $Y_{p_0}$. Moreover,
because $\widetilde p$ lies in the outer face of $Y_{p_0}$, we have
$$
\widetilde p\notin \Int(\gamma).
$$
Similarly, because $\widetilde q$ lies in the outer face of $g_0(Y_{p_0})$, we have
$$
\widetilde q\notin \Int(g_0(\gamma)).
$$
 This, together with the fact that $g_0$ satisfies Condition $(E3_F)$ implies that $\widetilde g_0$ also satisfies Condition $(E3_F)$.

By Theorem \ref{thm_main}, $\widetilde g_0$ admits a definable homeomorphic extension
$$
\widetilde G_0:\mathbb R^2\to\mathbb R^2.
$$
Define 
$$
 H(x)
:=\left\{
\begin{array}{rl}
\phi_{q_0}^{-1}\circ \widetilde G_0\circ\phi_{p_0},  & \text{if } x\neq p_0\\
 q_0,  & \text{if } x= p_0.
\end{array}
\right.
$$
Then $H$ is a definable homeomorphism
$
\mathbb R^2\to\mathbb R^2.
$

For every $x\in X$, we have
$$
H(x)
=
\phi_{q_0}^{-1}\circ \widetilde G_0\circ\phi_{p_0}(x)
=
\phi_{q_0}^{-1}\circ g_0\circ\phi_{p_0}(x)
=
h(x).
$$
Thus $H$ extends $h$.

Finally,
$$
H(p)
=
\phi_{q_0}^{-1}\bigl(\widetilde G_0(\phi_{p_0}(p))\bigr)
=
\phi_{q_0}^{-1}\bigl(\widetilde G_0(\widetilde p)\bigr)
=
\phi_{q_0}^{-1}(\widetilde q)
=
q.
$$
Therefore $H$ is a desired definable homeomorphic extension of $h$.
\end{proof}

Now we are ready to prove Theorem \ref{thm_main_unbounded}.

\begin{proof}[Proof of Theorem \ref{thm_main_unbounded}.]

Set
$$
Y_p:=\phi_p(X)\cup\{0\}.
$$
Since $h$ is coercive, the map
$$
g_{p,q}:Y_p\to \mathbb R^2
$$
defined by
$$
g_{p,q}(x)=\left\{
\begin{array}{rl}
\phi_q\circ h\circ \phi_p^{-1}(x), & \text{if }x\neq 0\\
0, & \text{if }x=0
\end{array}\right.
$$
is a definable embedding of the compact definable set $Y_p$.

Assume first that $h$ admits a definable homeomorphic extension to $\mathbb R^2$.
Then we have necessarily $\operatorname{ord}(h,\infty)\in\{\pm1\}$. By Lemma~4.2, we may choose
a definable homeomorphic extension
$$
H:\mathbb R^2\to \mathbb R^2
$$
of $h$ such that $H(p)=q$. Define
$$
G(x):=\left\{
\begin{array}{rl}
\phi_q\circ H\circ \phi_p^{-1}(x), & \text{if } x\ne 0\\     
0, & \text{if } x= 0\\
\end{array}
\right.
$$

Then $G$ is a definable homeomorphism of $\mathbb R^2$. 
Moreover, for every $x\in \phi_p(X)$, we have
$$
G(x)=\phi_q\circ h\circ \phi_p^{-1}(x)=g_{p,q}(x) \quad \text{ and }\quad G(0)=0=g_{p,q}(0).$$
Therefore $G$ is a definable homeomorphic extension
of $g_{p,q}$.

Conversely, suppose that $\operatorname{ord}(h,\infty)\in\{\pm1\}$ and that $g_{p,q}$
extends to a definable homeomorphism
$$
G:\mathbb R^2\to\mathbb R^2.
$$
Since $G$ extends $g_{p,q}$, we have $G(0)=0$. Define
$$
H:=\left\{
\begin{array}{rl}
\phi_q^{-1}\circ G\circ \phi_p, & \text{if } x\ne p\\
q , & \text{if } x= p.
\end{array}
\right.$$
Then $H$ is a definable homeomorphism of $\mathbb R^2$. 
Finally, for every $x\in X$,
$$
H(x)
=
\phi_q^{-1}\circ G\circ \phi_p(x)
=
\phi_q^{-1}\circ g_{p,q}\circ \phi_p(x)
=
h(x).
$$
Thus $H$ is a definable homeomorphic extension of $h$.
\end{proof}

\section{Examples}\label{section5}
In this section, we give examples showing that none of the conditions
$(E1)$, $(E2_F)$, and $(E3_F)$ in Theorem~\ref{thm_main} can be omitted.

\begin{example}\label{Ex1}\rm 
Let $X$ and $h(X)$ be as in Figure~\ref{Fig_E1} where $X$ is the black curve with six topological singular points $p_1,\dots,p_6$ and  $h: X \to \mathbb R^2$ is an embedding sending $X$ onto the black curve with singular points
$h(p_1),\dots,h(p_6)$.

\begin{figure}[htbp]
 \includegraphics[scale = 0.6]{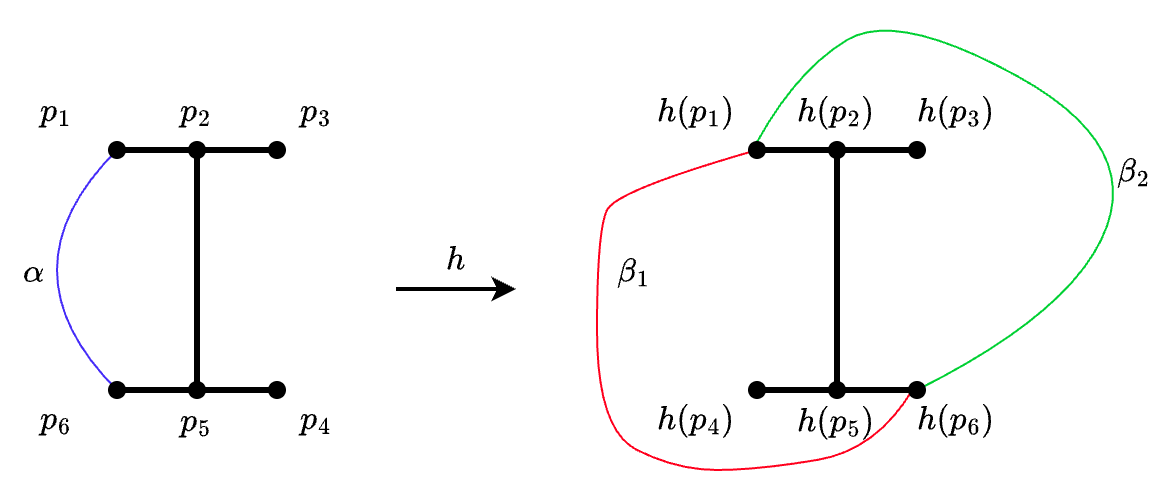}
\caption{}
\label{Fig_E1}
\end{figure}

Observe that in $X$, $p_2$ and $p_5$ are  singular points of degree $3$, and 
$$
\ord(h,p_2)=1 \quad \text{and} \quad \ord(h,p_5)=-1,
$$
so Condition $(E1)$  fails.

We claim that $h$ admits no homeomorphic extension to $\mathbb{R}^2$.
Indeed, suppose for contradiction that such an extension
$$
\widetilde h\colon \mathbb{R}^2 \to \mathbb{R}^2
$$
exists.
Let $\alpha$ be an arc joining $p_1$ and $p_6$ in $\mathbb{R}^2\setminus X$
(the blue arc in Figure \ref{Fig_E1}).
Then $\widetilde h(\alpha)$ is an arc joining $h(p_1)$ and $h(p_6)$ in
$\mathbb{R}^2\setminus h(X)$.

There are exactly two possible configurations for $\widetilde h(\alpha)$
(the red arc $\beta_1$ or the green arc $\beta_2$ in Figure \ref{Fig_E1}).
Note that $p_3$ and $p_4$ are contained in the exterior of Jordan curve
$p_1p_2p_5p_6p_1$.
However, for either choice of $\widetilde h(\alpha)$, one of the points
$h(p_3)$ or $h(p_4)$ is contained in the interior of the Jordan curve
$$
h(p_1)h(p_2)h(p_5)h(p_6)h(p_1)
= \widetilde h(p_1p_2p_5p_6p_1).
$$
This contradicts the fact that $\widetilde h$ is a homeomorphism of $\mathbb{R}^2$.
Hence $h$ does not admit a homeomorphic extension to $\mathbb{R}^2$.
\end{example}

\begin{example}\label{Example1.2}\rm
Let $X$ and $h(X)$ be as in Figure~\ref{Fig_E2} where $X$ consists of two circles
$\gamma_2<\gamma_1$, together with an arc joining them at points
$p_1\in \gamma_1$ and $p_2\in \gamma_2$, and  the definable embedding $h: X \to \mathbb R^2$ restricted 
to $\gamma_1$ is the identity map and, $h$ maps $\gamma_2$ onto a circle lying in the exterior of $h(\gamma_1)$ with the opposite orientation.

\begin{figure}[htbp]
 \includegraphics[scale = 0.6]{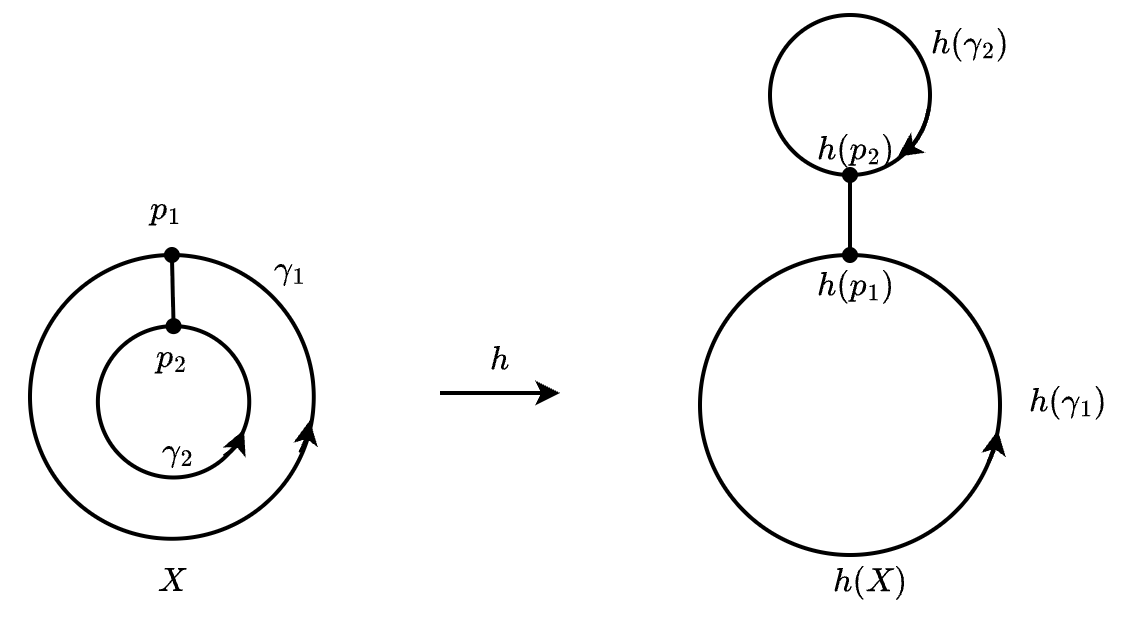}
\caption{}
\label{Fig_E2}
\end{figure}

It is clear that
$$
\ord(h,p_1)=\ord(h,p_2)=w(h,\gamma_2)=-1,
$$
whereas
$$
w(h,\gamma_1)=1.
$$
Thus Condition $(E1)$ holds, while Condition $(E2_F)$ fails.

We claim that $h$ cannot be extended to a homeomorphism of the whole plane
$\mathbb R^2$. Indeed, if such an extension existed, then it would have to map
the interior of $\gamma_1$ onto the interior of $h(\gamma_1)$. Since
$\gamma_2$ lies in the interior of $\gamma_1$, this would imply that
$h(\gamma_2)$ lies in the interior of $h(\gamma_1)$, contradicting the
assumption that $h(\gamma_2)$ lies in the exterior of $h(\gamma_1)$.
\end{example}

\begin{example}\label{Ex3}\rm 
Let $X$ and $h(X)$ be as in Figure~\ref{Fig_E3} where  $X=\gamma_1\cup\gamma_2$ is the union of two disjoint definable Jordan curves such that
$\gamma_2 <\gamma_1$, and  $h: X \to \mathbb R^2$ is a definable embedding sending $X$
onto two disjoint definable Jordan curves satisfying $h(\gamma_2) <h(\gamma_1)$ and
$$
w(h,\gamma_1)=1 \quad \text{and} \quad w(h,\gamma_2)=-1.
$$
\begin{figure}[htbp]
 \includegraphics[scale = 0.4]{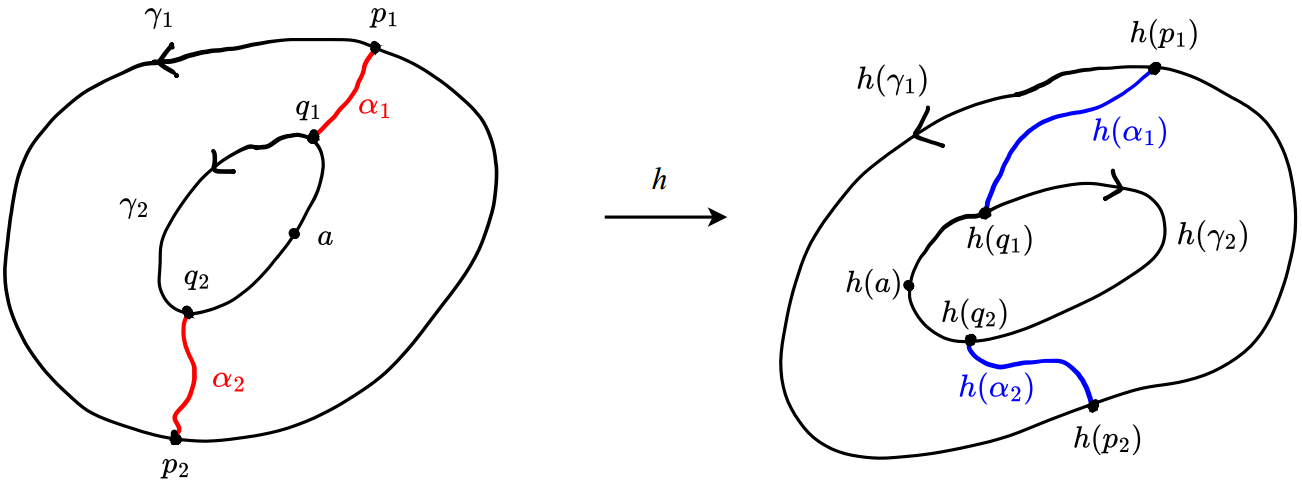}
\caption{}
\label{Fig_E3}
\end{figure}

It is obvious that Condition $(E3_F)$ holds while Condition $(E2_F)$ fails.
We will show that $h$ admits no homeomorphic extension to $\mathbb{R}^2$. Assume for contradiction that such an extension
$$
\widetilde h\colon \mathbb{R}^2 \to \mathbb{R}^2
$$
exists.

Choose distinct points $p_1,p_2\in\gamma_1$ and $q_1,q_2\in\gamma_2$.
Let $\alpha_i$, $i = 1, 2$ be definable arcs internally contained in
$\mathbb{R}^2\setminus X$ joining $p_i$ to $q_i$.
Consider the Jordan curve
$$
\beta = p_1p_2q_2q_1p_1,
$$
where $p_1p_2$ is the counterclockwise arc from $p_1$ to $p_2$ along $\gamma_1$ and
$q_2q_1$ is the clockwise arc from $q_2$ to $q_1$ along $\gamma_2$.

We have 
$$
\widetilde{h}(\beta) = h(p_1)h(p_2)h(q_2)h(q_1)h(p_1),
$$
where $h(p_1)h(p_2)$ is the counterclockwise arc on $h(\gamma_1)$ and
$h(q_2)h(q_1)$ is the counterclockwise arc on $h(\gamma_2)$.

Let $a \in [q_1,q_2]^{-}_{\gamma_2}$. 
Since $w(h,\gamma_2) = -1$,  $h(a)\in [h(q_1),h(q_2)]^{+}_{h(\gamma_2)} $.
Consequently, $a$ is contained in the exterior of the Jordan curve $\beta$
whereas $h(a)$ is contained in the interior of the Jordan curve $\widetilde{h}(\beta)$. This contradicts the fact that $\widetilde h$ is a homeomorphism of $\mathbb{R}^2$.
Therefore, $h$ does not admit a homeomorphic extension to $\mathbb{R}^2$.
\end{example}

\begin{example}\label{Ex4}\rm 
Let $X$ and $h(X)$ be as in Figure~\ref{Fig_E4} where $X$ is a  union of a definable Jordan curve $\gamma$ and a definable
one-dimensional set $Y \subset \Int(\gamma)$ consisting three topological singular points $p_1,p_2,p_3$ of degree $1$ and one topological singular point $p$ of degree $3$,  and $h\colon X\to \mathbb{R}^2$ is a definable embedding with $h(Y) \subset \Int(h(\gamma))$
and
$$
w(h,\gamma)=1 \quad \text{and} \quad \ord(h,p)=-1.
$$

\begin{figure}[htbp]
 \includegraphics[scale = 0.4]{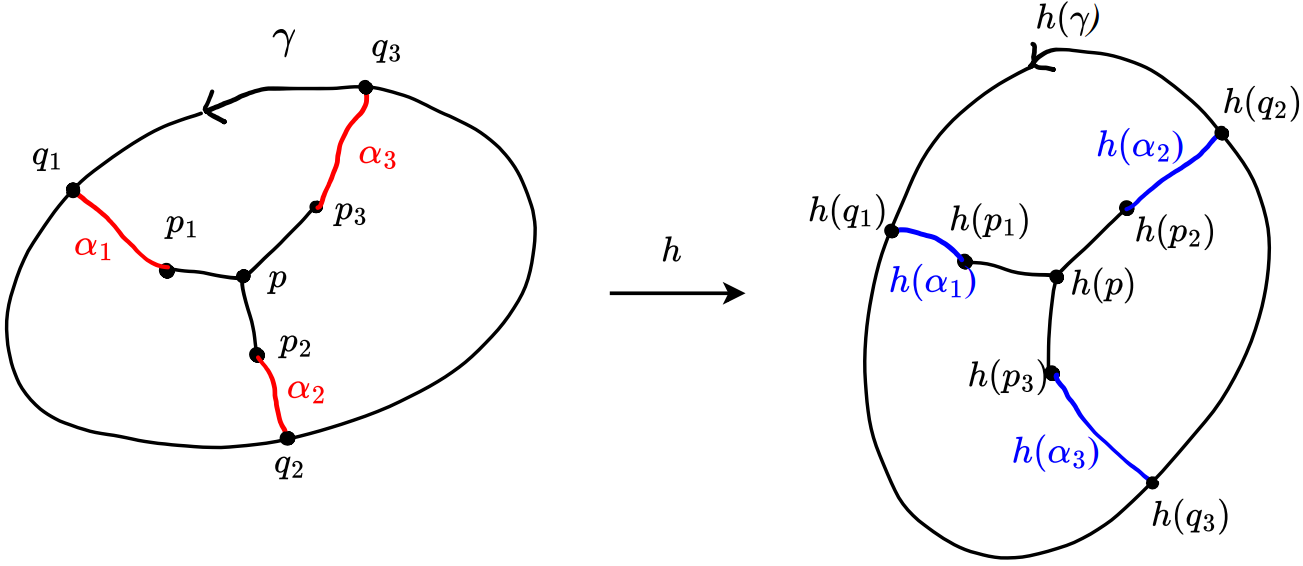}
\caption{}
\label{Fig_E4}
\end{figure}
Note that Conditions $(E3_F)$ is satisfied while $(E1)$ and $(E2_F)$ hold with different constants $c$. We claim that $h$ does not allow a homeomorphic extension to the whole of $\bb R^2$.

Indeed, assume for contradiction that $h$ admits a homeomorphic extension
to $\mathbb{R}^2$.
Let $X_i^p$ denote the germ at $p$ of the arc $[p,p_i]$, for $i=1,2,3$.
We see that 
$$
X_1^p < X_2^p < X_3^p < X_1^p.
$$
Let $\alpha_i$, $i=1,2,3$, be disjoint definable arcs internally contained in $\mathbb{R}^2\setminus X$  joining $p_i$ to points
$q_i\in\gamma$.
By construction, we have
$$
q_1 <_{\gamma} q_2 <_{\gamma} q_3 <_{\gamma} q_1.
$$
Hence $q_3$ is not contained in $[q_1q_2]^+_{\gamma}$.

Since $w(h,\gamma)=1$, the image $h([q_1q_2]^+_{\gamma}) = [h(q_1)h(q_2)]^+_{h(\gamma)}$.
On the other hand, the assumption $\ord(h,p)=-1$ implies that
$$
h(X_1^p) < h(X_3^p) < h(X_2^p) < h(X_1^p).
$$
Consequently,
$$
h(q_1) <_{h(\gamma)} h(q_3) <_{h(\gamma)} h(q_2) <_{h(\gamma)} h(q_1).
$$
Thus the arc $h([q_1q_2]^+_{\gamma})$ contains $h(q_3)$. This contradicts the fact that $q_3\notin [q_1q_2]^+_{\gamma}$.
Consequently, $h$ cannot admit a  homeomorphic extension to
$\mathbb{R}^2$.
\end{example}

\bibliographystyle{siam}
\bibliography{Biblio.bib}
\end{document}